\documentclass[11pt]{article}
\usepackage[left=1in,top=1in,right=1in,bottom=1in,letterpaper]{geometry}
\usepackage[ruled,vlined,linesnumbered]{algorithm2e}
\usepackage{amsmath, amsthm, amssymb, amsfonts, mathrsfs}
\usepackage{bbm}
\usepackage{bm}
\usepackage{color, colortbl}
\definecolor{LightCyan}{rgb}{0.88,1,1}
\usepackage{dirtytalk}
\usepackage{dsfont}
\usepackage{enumerate}
\usepackage{fullpage}
\usepackage{graphicx}
\usepackage{listings}
\usepackage{mathtools}
\usepackage{subfigure}
\usepackage{enumitem}
\usepackage{url}
\usepackage{xspace}
\usepackage{booktabs}
\usepackage{pifont}
\newcommand{\cmark}{\ding{51}}%
\newcommand{\xmark}{\ding{55}}%
\usepackage{makecell}
\usepackage{natbib}
\usepackage{multirow, array}

\usepackage[usenames,dvipsnames]{xcolor}

\usepackage[linktocpage,colorlinks,linkcolor=blue,anchorcolor=blue,citecolor=blue,urlcolor=blue,pagebackref]{hyperref}

\renewcommand*{\backref}[1]{\ifx#1\relax \else Page #1 \fi}
\renewcommand*{\backrefalt}[4]{%
  \ifcase #1 \footnotesize{(Not cited.)}%
  \or        \footnotesize{(Cited on page~#2.)}%
  \else      \footnotesize{(Cited on pages~#2.)}%
  \fi
}

\usepackage[capitalize,noabbrev]{cleveref}
\usepackage{tikz}

\newtheorem{theorem}{Theorem}
\newtheorem{lemma}{Lemma}

\newtheorem*{remark}{Remark}
\newtheorem{assumption}{Assumption}
\newtheorem{corollary}[theorem]{Corollary}
\newtheorem{definition}[theorem]{Definition}

\newcommand{\cD}{\mathcal{D}}

\newcommand{\cG}{\mathcal{G}}
\newcommand{\cH}{\mathcal{H}}

\newcommand{\cO}{\mathcal{O}}

\newcommand{\mA}{\mathbf{A}}
\newcommand{\mB}{\mathbf{B}}

\newcommand{\mI}{\mathbf{I}}

\newcommand{\mV}{\mathbf{V}}
\newcommand{\mW}{\mathbf{W}}
\newcommand{\mX}{\mathbf{X}}
\newcommand{\mY}{\mathbf{Y}}
\newcommand{\mZ}{\mathbf{Z}}
\newcommand{\mU}{\mathbf{U}}
\newcommand{\mF}{\mathbf{F}}

\def\setF{\mathscr{F}} 

\newcommand{\naturalset}{\mathbb{N}}
\newcommand{\realset}{\mathbb{R}}

\newcommand{\E}{\mathbb{E}}

\newcommand{\norm}[1]{\left\|#1\right\|}

\newcommand{\T}{^\top}
\newcommand{\avein}{\frac{1}{n}\sum_{i=1}^n}
\newcommand{\avejn}{\frac{1}{n}\sum_{j=1}^n}
\newcommand{\bfone}{\mathbf{1}}
\newcommand{\bfonet}{\mathbf{1}^{\top}}

\DeclareMathOperator*{\argmin}{arg\,min\,}

\DeclareMathOperator{\dom}{dom}

\DeclareMathOperator{\prox}{\mathbf{prox}}

\def\<#1,#2>{\left\langle #1,#2\right\rangle}
\mathchardef\mhyphen="2D
\allowdisplaybreaks

\title{\bf A One-Sample Decentralized Proximal Algorithm for Non-Convex Stochastic Composite Optimization \footnote{The first two authors contributed equally to this work.}}

\author{Tesi Xiao\thanks{Department of Statistics, UC Davis. \texttt{texiao@ucdavis.edu}. }
\and Xuxing Chen\thanks{Department of Mathematics, UC Davis. \texttt{xuxchen@ucdavis.edu}. }
\and Krishnakumar Balasubramanian\thanks{Department of Statistics, UC Davis. \texttt{kbala@ucdavis.edu}. Supported by NSF grant DMS-2053918.} 
\and Saeed Ghadimi\thanks{Department of Management Sciences, University of Waterloo. \texttt{sghadimi@uwaterloo.ca}. Paritally supported by NSERC grant RGPIN-2021-02644.}
}

\begin{document}

\maketitle

\begin{abstract}
We focus on decentralized stochastic non-convex optimization, where $n$ agents work together to optimize a composite objective function which is a sum of a smooth term and a non-smooth convex term. To solve this problem, we propose two single-time scale algorithms: \texttt{Prox-DASA} and \texttt{Prox-DASA-GT}. These algorithms can find $\epsilon$-stationary points in $\cO(n^{-1}\epsilon^{-2})$ iterations using constant batch sizes (i.e., $\mathcal{O}(1)$). Unlike prior work, our algorithms achieve comparable complexity without requiring large batch sizes, more complex per-iteration operations (such as double loops), or stronger assumptions. Our theoretical findings are supported by extensive numerical experiments, which demonstrate the superiority of our algorithms over previous approaches. Our code is available at \url{https://github.com/xuxingc/ProxDASA}.
\end{abstract}

\section{Introduction}
\label{sec:setting}

Decentralized optimization is a flexible paradigm for solving complex optimization problems in a distributed manner and has numerous applications in fields such as machine learning, robotics, and control systems. It has attracted increased attention due to the following benefits: (i) \textit{Robustness}: Decentralized optimization is more robust than centralized optimization because each agent can operate independently, making the system more resilient to failures compared to a centralized system where a coordinator failure or overload can halt the entire system. (ii) \emph{Privacy}: Decentralized optimization can provide greater privacy because each agent only has access to a limited subset of observations, which may help to protect sensitive information. (iii) \emph{Scalability}: Decentralized optimization is highly scalable as it can handle large datasets in a distributed manner, thereby solving complex optimization problems that are difficult or even impossible to solve in a centralized setting.

Specifically, we consider the following decentralized composite optimization problems in which $n$ agents collaborate to solve
\begin{equation}\label{eq:problem}
    \underset{x\in\realset^d}{\min}~ \Phi(x) \coloneqq  F(x) + \Psi(x), \ F(x) \coloneqq \frac{1}{n} \sum_{i=1}^{n} F_i(x),
\end{equation}
where each function $F_i(x)$ is a smooth function only known to the agent $i$; $\Psi(x)$ is non-smooth, convex, and shared across all agents; $\Phi(x)$ is bounded below by $\Phi_* > -\infty$. We consider the stochastic setting where the exact function values and derivatives of $F_{i}$'s are unavailable. In particular, we assume that $F_i(x)= \E_{\xi_i\sim \cD_i}[G_i(x, \xi_i)]$, where $\xi_i$ is a random vector and $\cD_i$ is the distribution used to generate samples for agent $i$. The agents form a connected and undirected network and can communicate with their neighbors to cooperatively solve \eqref{eq:problem}. The communication network can be represented with $\mathbb{G} = (\mathcal{V}, \mW)$ where $\mathcal{V} = \{v_1, v_2, \dots, v_n\}$ denotes all devices and $\mW = [w_{ij}]\in \realset^{n\times n}$ is the weighted adjacency matrix indicating how two agents are connected. 


A majority of the existing decentralized stochastic algorithms for solving~\eqref{eq:problem}, require large batch sizes to achieve convergence. 
The few algorithms that operate with constant batch sizes mainly rely on complicated variance reduction techniques and require stronger assumptions to establish convergence results. 
To the best of our knowledge, the question of whether it is possible to develop decentralized stochastic optimization algorithms to solve ~\eqref{eq:problem} without the above mentioned limitations, remains unresolved. 

To address this, we propose the two decentralized stochastic proximal algorithms, \texttt{Prox-DASA} and \texttt{Prox-DASA-GT}, for solving~\eqref{eq:problem} and make the following \textbf{contributions}:
\begin{itemize}[leftmargin=1em]
    \item We show that \texttt{Prox-DASA} is capable of achieving convergence in both homogenous and bounded heterogeneous settings while \texttt{Prox-DASA-GT} works for general decentralized heterogeneous problems.
    \item We show that both algorithms find an $\epsilon$-stationary point in $\cO(n^{-1}\epsilon^{-2})$ iterations using only $\cO(1)$ stochastic gradient samples per agent and $m$ communication rounds at each iteration, where $m$ can be any positive integer. A topology-independent transient time can be achieved by setting $m=\lceil\frac{1}{\sqrt{1-\rho}}\rceil$, where $\rho$ is the second-largest eigenvalue of the communication matrix.
    \item Through extensive experiments, we demonstrate the superiority of our algorithms over prior works. 
\end{itemize}
A summary of our results and comparison to prior work is provided in Table~\ref{tab:summary}.  


\textbf{Related Works on Decentralized Composite Optimization.} Motivated by wide applications in constrained optimization \citep{lee2013distributed, margellos2017distributed} and non-smooth problems with a composite structure as \eqref{eq:problem}, arising in signal processing \citep{ling2010decentralized, mateos2010distributed, patterson2014distributed} and machine learning \citep{facchinei2015parallel, hong2017prox}, several works have studied the decentralized composite optimization problem in~\eqref{eq:problem}, a natural generalization of smooth optimization. For example, \cite{shi2015proximal, li2019decentralized, alghunaim2019linearly, ye2020decentralized, xu2021distributed, li2021decentralized, sun2022distributed, wu2022unifying}  studied~\eqref{eq:problem} in the convex setting. Furthermore, \cite{facchinei2015parallel, di2016next, hong2017prox, zeng2018nonconvex, scutari2019distributed} studied~\eqref{eq:problem} in the deterministic setting. 

Although there has been a lot of research investigating decentralized composite optimization, the stochastic non-convex setting, which is more broadly applicable, still lacks a full understanding. \cite{wang2021distributed} proposes \texttt{SPPDM}, which uses a proximal primal-dual approach to achieve $\cO(\epsilon^{-2})$ sample complexity. \texttt{ProxGT-SA} and \texttt{ProxGT-SR-O} \citep{xin2021stochastic} incorporate stochastic gradient tracking and multi-consensus update in proximal gradient methods and obtain $\cO(n^{-1}\epsilon^{-2})$ and $\cO(n^{-1}\epsilon^{-1.5})$ sample complexity respectively, where the latter further uses a \texttt{SARAH} type variance reduction method \citep{pham2020proxsarah, wang2019spiderboost}. A recent work \citep{mancino2022proximal} proposes \texttt{DEEPSTORM}, which leverages the momentum-based variance reduction technique and gradient tracking to obtain $\cO(n^{-1}\epsilon^{-1.5})$ and $\tilde{\cO}(\epsilon^{-1.5})$ sample complexity under different stepsize choices. Nevertheless, existing works either require stronger assumptions \citep{mancino2022proximal} or increasing batch sizes \citep{wang2021distributed, xin2021stochastic}.



\begin{table*}[t]
\centering
\renewcommand{\arraystretch}{2}
\caption{Comparison of decentralized proximal gradient based algorithms to find an $\epsilon$-stationary solution to stochastic composite optimization in the nonconvex setting. The sample complexity is defined as the number of required samples per agent to obtain an $\epsilon$-stationary point (see Definition \ref{def: stat&cons}). We omit a comparison with \texttt{SPPDM} \citep{wang2021distributed} as their definition of stationarity differs from ours; see Appendix for further discussions.} 
\label{tab:summary}
\resizebox{\textwidth}{!}{%
\begin{tabular}{| c | c| c| c| c| c| c|}
\hline
 \textbf{Algorithm}   & \makecell{\bf Batch Size} & \makecell{\bf Sample \\ \bf Complexity } & \makecell{\bf Communication \\ \bf Complexity} & \makecell{\bf Linear \\ \bf Speedup?} & \bf Remark\\
\hline
 \makecell{\texttt{ProxGT-SA}\\ \citep{xin2021stochastic}} &  $\cO(\epsilon^{-1})$ & $\cO(n^{-1}\epsilon^{-2})$ & $\mathcal{O}(\log (n) \epsilon^{-1})$ & \cmark & \\
 \hline
 \makecell{\texttt{ProxGT-SR-O}\\ \citep{xin2021stochastic}} &  $\cO(\epsilon^{-1})$ & $\cO(n^{-1}\epsilon^{-1.5})$ & $\mathcal{O}(\log (n) \epsilon^{-1})$ & \cmark & \makecell{double-loop; \\ mean-squared smoothness}\\
 \hline
 \multirow{2}{*}{\makecell{\texttt{DEEPSTORM} \\ \citep{mancino2022proximal}}} &  $\cO(\epsilon^{-0.5})$ then $\cO(1) ^ *$ & $\cO(n^{-1}\epsilon^{-1.5})$ & $\cO(n^{-1}\epsilon^{-1.5})$ & \cmark & \multirow{2}{*}{\makecell{two time-scale; \\mean-squared smoothness; \\ double gradient evaluations \\ per iteration}}\\ \cline{2-5}
 &  $\cO(1)$ & $\cO(\epsilon^{-1.5}|\log\epsilon|^{-1.5})$ & $\cO(\epsilon^{-1.5}|\log\epsilon|^{-1.5})$ & \xmark & \\
 \hline
   \texttt{Prox-DASA}  (Alg. \ref{algo: Prox-DASA}) & $\mathcal{O}\left(1\right)$ &  $\mathcal{O}(n^{-1}\epsilon^{-2})$ &  $\mathcal{O}(n^{-1} \epsilon^{-2})$ & \cmark & \makecell{bounded heterogeneity}\\
 \hline
\rowcolor{LightCyan}
   \texttt{Prox-DASA-GT} (Alg. \ref{algo: Prox-DASA-GT})  & $\mathcal{O}\left(1\right)$ &  $\mathcal{O}(n^{-1}\epsilon^{-2})$ &  $\mathcal{O}(n^{-1}\epsilon^{-2})$ & \cmark & \\
\hline
\end{tabular}
}
\footnotesize{$^*$ It requires $\cO(\epsilon^{-0.5})$ batch size in the first iteration and then $\cO(1)$ for the rest (see $m_0$ in Algorithm 1 in \cite{mancino2022proximal}).}
\end{table*}

\section{Preliminaries}

\textbf{Notations.} $\|\cdot\|$ denotes the $\ell_2$-norm for vectors and Frobenius norm for matrices. $\|\cdot\|_2$ denotes the spectral norm for matrices. $\mathbf{1}$ represents the all-one vector, and $\mI$ is the identity matrix as a standard practice. We identify vectors at agent $i$ in the subscript and use the superscript for the algorithm step. For example, the optimization variable of agent $i$ at step $k$ is denoted as $x^k_i$, and $z^k_i$ is the corresponding dual variable. We use uppercase bold letters to represent the matrix that collects all the variables from nodes (corresponding lowercase) as columns. We add an overbar to a letter to denote the average over all nodes. For example, we denote the optimization variables over all nodes at step $k$ as
$\mX_k = \left[x_{1}^{k}, \dots, x_{n}^{k}\right].$
The corresponding average over all nodes can be thereby defined as
\begin{align*}
\bar x^k &= \avein x_{i}^{k} = \frac{1}{n}\mX_k \bfone,\\
\bar \mX_k &= [\bar x^k, \dots, \bar x^k] = \bar x^k \bfonet = \frac{1}{n}\mX_k \bfone \bfonet.
\end{align*}
For an extended valued function $\Psi: \realset^d \rightarrow \realset \cup \{+\infty\}$, its effective domain is written as $\dom(\Psi) = \{x \mid \Psi(x)<+\infty\}$. A function $\Psi$ is said to be proper if $\dom(\Psi)$ is non-empty. For any proper closed convex function $\Psi$, $x\in\realset^d$, and scalar $\gamma>0$, the proximal operator is defined as
\[
\prox_{\Psi}^{\gamma}(x) = \argmin_{y\in \realset^d}\left\{\frac{1}{2\gamma}\|y - x\|^2  + \Psi(y)\right\} .
\]
For $x, z\in \realset^d$ and $\gamma > 0$, the proximal gradient mapping of $z$ at $x$ is defined as
\[
\cG(x, z, \gamma) = \frac{1}{\gamma}\left( x -  \prox_{\Psi}^{\gamma}(x-\gamma z)\right).
\]
All random objects are properly defined in a probability space $(\Omega, \setF, \mathbb{P})$ and write  $x \in \cH$ if $x$ is $\cH$-measurable given a sub-$\sigma$-algebra $\cH \subseteq \setF$ and a random vector $x$. We use $\sigma(\cdot)$ to denote the $\sigma$-algebra generated by all the argument random vectors.

\textbf{Assumptions.} Next, we list and discuss the assumptions  made in this work.

\begin{assumption}\label{aspt:gossipMatrix}
The weighted adjacency matrix $\mW=(w_{ij})\in\realset^{n\times n}$ is symmetric and doubly stochastic, i.e., $$ \mW = \mW^\top,\ \mW \mathbf{1}_n = \mathbf{1}_n,\  w_{ij}\geq 0,\ \forall i, j, $$ and its eigenvalues satisfy $1=\lambda_1 > \lambda_2 \geq \dots \geq \lambda_n$ and $\rho\coloneqq \max\{|\lambda_2|, |\lambda_n|\}<1$.
\end{assumption}

\begin{assumption}\label{aspt:lipschitz-gradient}
All functions $\{F_{i}\}_{1\leq i\leq n}$ have Lipschitz continuous gradients with Lipschitz constants $L_{\nabla F_{i}}$, respectively. Therefore, $\nabla F$ is $L_{\nabla F}$-Lipchitz continous with $L_{\nabla F} ={\max}_{1\leq i\leq n} \{L_{\nabla F_i}\}$.
\end{assumption}

\begin{assumption}\label{aspt:Psi}
The function $\Psi: \realset^d \rightarrow \realset\cup\{+\infty\}$ is a closed proper convex function.
\end{assumption}

For stochastic oracles, we assume that each node $i$ at every iteration $k$ is able to obtain a local random data vector $\xi^{k}_i$. The induced natural filtration is given by $\setF_0 = \{\emptyset, \Omega\}$ and 
\[
    \setF_k \coloneqq \sigma\left(\xi^{t}_i \mid i =1 ,\dots, n, \, t=1,\dots, k \right), \forall k\geq 1.
\]
We require that the stochastic gradient $\nabla G_i(\cdot, \xi^{k+1}_i)$ is unbiased conditioned on the filteration $\setF_k$. 
\begin{assumption}[Unbiasness]\label{aspt: Unbiasness} For any $k\geq 0, x\in \setF_k$, and $1\leq i\leq n$, $$\E\left[\nabla G_i(x, \xi^{k+1}_i)\mid \setF_k\right] = \nabla F_{i}(x).$$
\end{assumption}
\begin{assumption}[Independence]\label{aspt: independence} For any $k\geq 0, 1\leq i, j\leq n, i\neq j,\ \xi_i^{k+1}$ is independent of $\setF_k$, and $\xi_i^{k+1}$ is independent of $\xi_j^{k+1}$.
\end{assumption}
In addition, we consider two standard assumptions on the variance and heterogeneity of stochastic gradients.
\begin{assumption}[Bounded variance]\label{aspt: Bounded Variance} For any $k\geq 0, x\in \setF_k$, and $1\leq i\leq n$, $$\E\left[\norm{\nabla G_i(x, \xi^{k+1}_i) - \nabla F_{i}(x)}^2\middle\vert \setF_k\right] \leq \sigma^2_i.$$ Let $\sigma^2 = \frac{1}{n}\sum_{i=1}^{n} \sigma_i^2$.
\end{assumption}

\begin{assumption}[Gradient heterogeneity]\label{aspt: Gradient heterogeneity} There exists a constant $\nu\geq0$ such that for all $1 \leq
i \leq n, x\in \realset^d$, $$\norm{\nabla F_i(x) - \nabla F(x)} \leq \nu.$$
\end{assumption}

\begin{remark}[Bounded gradient heterogeneity]
 The above assumption of gradient heterogeneity is standard \cite{lian2017can} and less strict than the bounded second moment assumption on stochastic gradients which implies lipschtizness of functions $\{F_i\}$. However, this assumption is only required for the convergence analysis of \texttt{Prox-DASA} and can be bypassed by employing a gradient tracking step. 
\end{remark}

\begin{remark}[Smoothness and mean-squared smoothness]
Our theoretical results of the proposed methods are only built on the smoothness assumption on functions $\{F_i\}$ without further assuming mean-squared smoothness assumptions on $\{G_{i,\xi}\}$, which is required in all variance reduction based methods in the literature, such as \texttt{ProxGT-SR-O}~\citep{xin2021stochastic} and \texttt{DEEPSTORM}~\citep{mancino2022proximal}. It is worth noting that a clear distinction in the lower bounds of sample complexity for solving stochastic optimization under two different sets of assumptions has been proven in \citep{arjevani2023lower}. Specifically, when considering the mean-squared smoothness assumption, the optimal sample complexity is $\mathcal{O}(\epsilon^{-1.5})$, whereas under smoothness assumptions, it is $\mathcal{O}(\epsilon^{-2})$. The proposed methods in this work achieve the optimal sample complexity under our weaker assumptions.
\end{remark}

\section{Algorithm}

Several algorithms have been developed to solve Problem \eqref{eq:problem} in the stochastic setting; see Table \ref{tab:summary}. However, the most recent two types of algorithms have certain drawbacks: (i) \textbf{increasing batch sizes}: \texttt{ProxGT-SA}, \texttt{Prox-SR-O}, and \texttt{DEEPSTORM} with constant step sizes (Theorem 1 in \citep{mancino2022proximal}) require batches of stochastic gradients with batch sizes inversely proportional to tolerance $\epsilon$; (ii) \textbf{algorithmic complexities}: \texttt{ProxGT-SR-O} and \texttt{DEEPSTORM} are either double-looped or two-time-scale, and require stochastic gradients evaluated at different parameter values over the same sample, i.e., $\nabla G_i(x, \xi)$ and $\nabla G_i(x', \xi)$. These variance reduction techniques are unfavorable when gradient evaluations are computationally expensive such as forward-backward steps for deep neural networks. (iii) \textbf{theoretical weakness}: the convergence analyses of \texttt{ProxGT-SR-O} and \texttt{DEEPSTORM} are established under the \emph{stronger} assumption of mean-squared lipschtizness of stochastic gradients. In addition, Theorem 2 in \citep{mancino2022proximal} fails to provide linear-speedup results for one-sample variant of \texttt{DEEPSTORM} with diminishing stepsizes.

\subsection{Decentralized Proximal Averaged Stochastic Approximation}

To address the above limitations, we propose \textbf{D}ecentralized \textbf{Prox}imal \textbf{A}veraged \textbf{S}tochastic \textbf{A}pprox-imation (\texttt{Prox-DASA}) which leverages a common averaging technique in stochastic
optimation \citep{ruszczynski2008merit, mokhtari2018conditional, ghadimi2020single} to reduce the error of gradient estimation. In particular, the sequences of dual variables $\mZ^k = [ z_1^k, \dots, z_n^k]$ that aim to approximate gradients are defined in the following recursion:
\begin{equation*}
\begin{split}
    \mZ^{k+1} &= \left\{(1-\alpha_k) \mZ^{k} + \alpha_k \mV^{k+1}\right\} \mW^m\\
    \mV^{k+1} &= [v_1^{k+1}, \dots, v_n^{k+1}],
\end{split}
\end{equation*}
where each $v_i^{k+1}$ is the local stochastic gradient evaluated at the local variable $x_i^k$. For complete graphs where each pair of graph vertices is connected by an edge and there is no consensus error for optimization variables, i.e., $\mW=\frac{1}{n}\bfone\bfonet$ and $x_i^k = x_j^k, \forall i,j$, the averaged dual variable over nodes $\bar z^k$ follows the same averaging rule as in centralized algorithms:
\begin{equation*}
\begin{split}
    \bar z^{k+1} &= (1-\alpha_k) \bar z^{k} + \alpha_k \bar v^{k+1}\\
    \E[\bar v^{k+1} |\setF_k] &= \nabla F(\bar x^k).
\end{split}
\end{equation*}
To further control the consensus errors, we employ a multiple consensus step for both primal and dual iterates $\{x^k_i, z^k_i\}$ which multiply the matrix of variables from all nodes by the weight matrix $m$ times. A pseudo code of \texttt{Prox-DASA} is given in Algorithm \ref{algo: Prox-DASA}.
\begin{algorithm}[t]
    \caption{\texttt{Prox-DASA}}\label{algo: Prox-DASA}
    \SetAlgoLined
    \KwIn{$x_i^0 = z_i^0 = \mathbf{0}, \gamma, \{\alpha_k\}_{\geq 0}, m$}
    \For{$k=0, 1,\dots,K-1$}{
        \CommentSty{\# Local Update}\\
        \For{$i=1,2,\dots,n$ (in parallel)}{
            $y_i^k = \prox_{\Psi}^{\gamma}\left(x_{i}^{k} - \gamma z_{i}^{k}\right)$\\
            $\tilde{x}_{i}^{k+1} = (1- \alpha_k)x_{i}^{k} + \alpha_ky_{i}^{k}$\\
            \CommentSty{\# Compute stochastic gradient}\\
            $v_{i}^{k+1} = \nabla G_{i}(x_{i}^{k}, \xi_{i}^{k+1})$\\
            $\tilde{z}_{i}^{k+1} = (1 - \alpha_k)z_{i}^{k} + \alpha_k v_{i}^{k+1}$\\
        }
        \CommentSty{\# Communication}\\
        $[x_1^{k+1}, \dots, x_{n}^{k+1}] = [\tilde{x}_1^{k+1}, \dots, \tilde{x}_{n}^{k+1}]\mW^m$\\
        $[z_1^{k+1}, \dots, z_{n}^{k+1}] = [\tilde{z}_1^{k+1}, \dots, \tilde{z}_{n}^{k+1}]\mW^m$
    }
\end{algorithm}

\subsection{Gradient Tracking}

The constant $\nu$ defined in Assumption \ref{aspt: Gradient heterogeneity} measures the heterogeneity between local gradients and global gradients, and hence the variance of datasets of different agents. To remove $\nu$ in the complexity bound, \cite{tang2018d} proposed the $\text{D}^2$ algorithm, which modifies the $x$ update in D-PSGD \citep{lian2017can}. However, it requires one additional assumption on the eigenvalues of the mixing matrix $\mW$. Here we adopt the gradient tracking technique, which was first introduced to deterministic distributed optimization to improve the convergence rate \citep{xu2015augmented, di2016next, nedic2017achieving, qu2017harnessing}, and was later proved to be useful in removing the data variance (i.e., $\nu$) dependency in the stochastic case \citep{zhang2019decentralized, lu2019gnsd, pu2021distributed, koloskova2021improved}. In the convergence analysis of \texttt{Prox-DASA}, an essential step is to control the heterogeneity of stochastic gradients, i.e., $\E[\norm{\mV^{k+1} - \bar \mV^{k+1}}^2]$, which requires bounded heterogeneity of local gradients (Assumption \ref{aspt: Gradient heterogeneity}). To pypass this assumption, we employ a gradient tracking step by replacing $\mV^{k+1}$ with pseudo stochastic gradients $\mU^{k+1} = [u_1^{k+1}, \dots, u_n^{k+1}]$, which is updated as follows:
\begin{equation*}
    \mU^{k+1} = \left(\mU^k + \mV^{k+1} - \mV^{k}\right) \mW^m.
\end{equation*}
Provided that $\mU^0 = \mV^0$ and $\mW\bfone = \bfone$, one can show that $\bar u^k = \bar v^k$ at each step $k$. In addition, with the consensus procedure over $\mU^k$, the heterogeneity of pseudo stochastic gradients $\E[\norm{\mU^{k+1} - \bar \mU^{k+1}}^2]$ can be bounded above. The proposed algorithm, named as \texttt{Prox-DASA} with Gradient Tracking (\texttt{Prox-DASA-GT}), is presented in Algorithm \ref{algo: Prox-DASA-GT}.
\begin{algorithm}[t]
    \caption{\texttt{Prox-DASA-GT}}\label{algo: Prox-DASA-GT}
    \SetAlgoLined
    \KwIn{$x_i^0 = z_i^0 = u_i^0 = \mathbf{0}, \gamma, \{\alpha_k\}_{\geq 0}, m$}
    \For{$k=0, 1,\dots,K$}{
        \CommentSty{\# Local Update}\\
        \For{$i=1,2,\dots,n$ (in parallel)}{
            $y_{i}^{k} = \prox_{\Psi}^{\gamma}\left(x_{i}^{k} - \gamma z_{i}^{k}\right)$\\
            $\tilde{x}_{i}^{k+1} = (1- \alpha_k)x_{i}^{k} + \alpha_ky_{i}^{k}$\\
            \CommentSty{\# Compute stochastic gradient}\\
            $v_{i}^{k+1} = \nabla G_{i}(x_{i}^{k}, \xi_{i}^{k+1})$\\
            $\tilde u_i^{k+1} = u_i^k + v_i^{k+1} - v_i^k$ \\
            $\tilde{z}_{i}^{k+1} = (1 - \alpha_k)z_{i}^{k} + \alpha_k \tilde{u}_{i}^{k+1}$\\
        }
        \CommentSty{\# Communication}\\
        $[x_1^{k+1}, \dots, x_{n}^{k+1}] = [\tilde{x}_1^{k+1}, \dots, \tilde{x}_{n}^{k+1}]\mW^m$\\
        $[u_1^{k+1}, \dots, u_{n}^{k+1}] = [\tilde{u}_1^{k+1}, \dots, \tilde{u}_{n}^{k+1}]\mW^m$\\
        $[z_1^{k+1}, \dots, z_{n}^{k+1}] = [\tilde{z}_1^{k+1}, \dots, \tilde{z}_{n}^{k+1}]\mW^m$
    }
\end{algorithm}

\subsection{Consensus Algorithm}
In practice, we can leverage accelerated consensus algorithms, e.g., \cite{liu2011accelerated, olshevsky2017linear}, to speed up the multiple consensus step $\mW^m$ to achieve improved communication complexities when $m>1$. Specifically, we can replace $\mW^m$ by a Chebyshev-type polynomial of $\mW$ as described in Algorithm \ref{algo: acc-consensus}, which can improve the $\rho$-dependency of the communication complexity from a factor of $\frac{1}{1-\rho}$ to $\frac{1}{\sqrt{1-\rho}}$.

\begin{algorithm}[th]
    \caption{Chebyshev Mixing Protocol}\label{algo: acc-consensus}
    \SetAlgoLined
    \KwIn{Matrix $\mX$, mixing matrix $\mW$, rounds $m$}
    Set $\mA_0= \mX, \mA_1 = \mX\mW, \rho=\max\{|\lambda_2(\mW)|, |\lambda_n(\mW)|\}<1, \mu_0 = 1, \mu_1 = \frac{1}{\rho}$\\
    \For{$t=1,\dots,m-1$}{
        $\mu_{t+1} = \frac{2}{\rho}\mu_t - \mu_{t-1}$\\
        $\mA_{t+1} = \frac{2\mu_t}{\rho\mu_{t+1}}\mA_t \mW - \frac{\mu_{t-1}}{\mu_{t+1}} \mA_{t-1}$
    }
    \KwOut{$\mA_m$}
\end{algorithm}

\section{Convergence Analysis}

\subsection{Notion of Stationarity}

For centralized optimization problems with non-convex objective function $F(x)$, a standard measure of non-stationarity of a point $\bar x$ is the squared norm of proximal gradient mapping of $\nabla F(\bar x)$ at $\bar x$, i.e.,
\begin{equation*}
\norm{\cG(\bar x, \nabla F(\bar x), \gamma)}^2 = \norm{\frac{1}{\gamma}\left(x - \prox_{\Psi}^{\gamma}(\bar x-\gamma \nabla F(\bar x))\right)}^2.
\end{equation*}
For the smooth case where $\Psi(x)\equiv 0$, the above measure is reduced to $\norm{\nabla F(\bar x)}^2$.

However, in the decentralized setting with a connected network $\mathbb{G}$, we solve the following equivalent reformulated consensus optimization problem:
\begin{equation}\label{eq:problem-consensus}
\begin{split}
    \underset{x_1,\dots,x_n\in\realset^d}{\min}& \quad \frac{1}{n} \sum_{i=1}^{n}\left\{ F_i(x_i) + \Psi(x_i)\right\}\\
    \text{s.t.}\quad &\quad    x_i = x_j,\ \forall (i,j).
\end{split}
\end{equation}
To measure the non-stationarity in Problem \eqref{eq:problem-consensus},  one should consider not only the stationarity violation at each node but also the consensus errors over the network. Therefore, \citet{xin2021stochastic} and \citet{mancino2022proximal} define an $\epsilon$-stationary point $\mX= [x_1, \dots, x_n]$ of Problem \ref{eq:problem-consensus} as
\begin{equation}\label{def: previous stationarity}
\E\left[\frac{1}{n}\sum_{i=1}^{n}\left\{\norm{\cG(x_i, \nabla F(x_i), \gamma)}^2 + L_{\nabla F}^2\norm{x_i - \bar x}^2\right\}\right] \leq \epsilon.    
\end{equation}
In this work, we use a general measure as follows.
\begin{definition}\label{def: stat&cons}
    Let $\mX = [x_1, \dots, x_n]$ be random vectors generated by a decentralized algorithm to solve Problem \ref{eq:problem-consensus} and $\bar x = \frac{1}{n}\sum_{i=1}^{n}x_i$. We say that $\mX$ is an $\epsilon$-stationary point of Problem \ref{eq:problem-consensus} if 
    \begin{align}
    &\text{(stationarity violation)} &\E\left[\norm{\cG(\bar x, \nabla F(\bar x), \gamma)}^2\right] \leq \epsilon, \notag\\
    &\text{(consensus error)} &\textstyle \E\left[\frac{L_{\nabla F}^2}{n}\norm{\mX - \bar \mX}^2\right] \leq \epsilon \notag.
    \end{align}
\end{definition}
The next inequality characterizes the difference between the gradient mapping at $\bar x$ and $x_i$, which relates our definition to \eqref{def: previous stationarity}. Noting that by non-expansiveness of the proximal operator, we have $\norm{\cG(x_i, \nabla F(x_i), \gamma) - \cG(\bar x, \nabla F(\bar x), \gamma)}\leq \tfrac{2+\gamma L_{\nabla F}}{\gamma}\norm{x_i - \bar x}$, implying 
\begin{align*}
    \frac{1}{n}\sum_{i=1}^{n}\norm{\cG(x_i, \nabla F(x_i), \gamma)}^2 \lesssim \norm{\cG(\bar x, \nabla F(\bar x), \gamma)}^2 + \tfrac{1}{\gamma^2 n}\norm{\mX - \bar \mX}^2.
\end{align*}

\subsection{Main Results}

We present the complexity results of our algorithms below.

\begin{theorem}\label{thm: main}
    Suppose Assumptions \ref{aspt:gossipMatrix}, \ref{aspt:lipschitz-gradient}, \ref{aspt:Psi}, \ref{aspt: Unbiasness}, \ref{aspt: independence}, \ref{aspt: Bounded Variance} hold and the total number of iterations $K\geq K_0$, where $K_0$ is a constant that only depends on constants $(n, L_{\nabla F}, \varrho(m), \gamma)$, where $\varrho(m) = \tfrac{(1+\rho^{2m})\rho^{2m}}{(1-\rho^{2m})^2}$. Let $C_0$ be some initialization-dependent constant and $R$ be a random integer uniformly distributed over $\{1, 2, \dots, K\}$. Suppose we set $\alpha_k \asymp \sqrt{\tfrac{n}{K}}, \gamma \asymp \tfrac{1}{L_{\nabla F}}$.
    \begin{itemize}[leftmargin=0pt]
        \item[] \textbf{(Prox-DASA)} Suppose in addition Assumption \ref{aspt: Gradient heterogeneity} also holds. The, for  Algorithm \ref{algo: Prox-DASA} we have
        \begin{align*}
            &\E\left[\norm{\cG(\bar x^R, \nabla F(\bar x^R), \gamma)}^2\right] \lesssim \frac{\gamma^{-1} C_0 + \sigma^2}{\sqrt{nK}} + \frac{n(\sigma^2 +\gamma^{-2}\nu^2)\varrho(m)}{K},\\
            &\E\left[\norm{\bar z^R - \nabla F(\bar x^R))}^2\right] \lesssim \frac{\gamma^{-1} C_0 + \sigma^2}{\sqrt{nK}} + \frac{n(\sigma^2 +\gamma^{-2}\nu^2)\varrho(m)}{K},\\
            &\E\left[\frac{L_{\nabla F}^2}{n}\norm{\mX_R - \bar \mX_R}^2+ \frac{1}{n}\norm{\mZ_R - \bar \mZ_R}^2\right] \lesssim \frac{n(\sigma^2 +\gamma^{-2}\nu^2)\varrho(m)}{K}.
        \end{align*}

        \item[] \textbf{(Prox-DASA-GT)} For Algorithm \ref{algo: Prox-DASA-GT} we have 
        \begin{align*}
            &\E\left[\norm{\cG(\bar x^R, \nabla F(\bar x^R), \gamma)}^2\right] \lesssim \frac{\gamma^{-1} C_0 + \sigma^2}{\sqrt{nK}} + \frac{n\sigma^2\varrho(m)}{K},\\
            &\E\left[\norm{\bar z^R - \nabla F(\bar x^R)}^2\right] \lesssim \frac{\gamma^{-1} C_0 + \sigma^2}{\sqrt{nK}} + \frac{n\sigma^2\varrho(m)}{K},\\
            &\E\left[\frac{L_{\nabla F}^2}{n}\norm{\mX_R - \bar \mX_R}^2 + \frac{1}{n}\norm{\mZ_R - \bar \mZ_R}^2\right] \lesssim \frac{n\sigma^2\varrho(m)}{K}.
        \end{align*}
    \end{itemize}
\end{theorem}
In Theorem \ref{thm: main} for simplicity we assume $\gamma \asymp \frac{1}{L_{\nabla F}}$, which can be relaxed to $\gamma > 0$. We have the following corollary characterizing the complexity of Algorithm \ref{algo: Prox-DASA} and \ref{algo: Prox-DASA-GT} for finding $\epsilon$-stationary points. The proof is immediate.
\begin{corollary}\label{cor: complexity}
    Under the same conditions of Theorem \ref{thm: main}, provided that $K\gtrsim n^3 \varrho(m)$, for any $\epsilon > 0$ the sample complexity per agent for finding $\epsilon$-stationary points
    in Algorithm \ref{algo: Prox-DASA} and \ref{algo: Prox-DASA-GT} are $\cO(\max\{n^{-1}\epsilon^{-2}, K_T\})$ where the transient time $K_T \asymp \max\{K_0, n^3 \varrho(m)\}$.
\end{corollary}

\begin{remark}[Sample complexity]
    For a sufficiently small $\epsilon>0$, Corrollary \ref{cor: complexity} implies that the sample complexity of Algorithm \ref{algo: Prox-DASA} and \ref{algo: Prox-DASA-GT}  matches the optimal lower bound $\cO(n^{-1}\epsilon^{-2})$ in decentralized \emph{smooth} stochastic non-convex optimization \citep{lu2021optimal}.
\end{remark}

\begin{remark}[Transient time and communication complexity]
    Our algorithms can achieve convergence with a single communication round per iteration, i.e., $m=1$, leading to a topology-independent $\cO(n^{-1}\epsilon^{-2})$ communication complexity. In this case, however, the transient time $K_T$ still depends on $\rho$, as is also the case for smooth optimization problems \citep{xin2021improved}. When considering multiple consensus steps per iteration with the communication complexity being $\cO(mn^{-1}\epsilon^{-2})$, setting $m\asymp \lceil \tfrac{1}{1-\rho}\rceil$ (or $m\asymp \lceil \tfrac{1}{\sqrt{1-\rho}}\rceil$ for accelerated consensus algorithms) results in a topology-independent transient time given that $\varrho(m) \asymp 1$.
\end{remark}

\begin{remark}[Dual convergence]
    An important aspect to emphasize is that in our proposed methods, the sequence of average dual variables $\bar z^k = \frac{1}{n}\sum_{i=1}^{n} z_i^k$ converges to $\nabla F(\bar x^k)$, while the consensus error of $\{z_1^k, \dots, z_n^k\}$ decreases to zero. Our approach achieves this desirable property, which is commonly observed in modern variance reduction methods~\citep{gower2020variance}, without the need for complex variance reduction operations in each iteration. As a result, it provides a reliable termination criterion in the stochastic setting without requiring large batch sizes.
\end{remark}

\subsection{Proof Sketch}
\label{sec: proof_sketch}
Here, we present a sketch of our convergence analyses and defer details to Appendix. Our proof relies on the merit function below:
\begin{equation*}
\begin{split}
    W(\bar x^k,\bar z^k) = \underbrace{\Phi(\bar x^{k}) - \Phi_*}_{\text{function value gap}} + \underbrace{\Psi(\bar x^k) - \eta(\bar x^{k}, \bar z^{k})}_{\text{primal convergence}} + \lambda \underbrace{\norm{\nabla F(\bar x^k) - \bar z^k}^2}_{\text{dual convergence}},
\end{split}
\end{equation*}
where $\eta(x, z) = \underset{y\in \realset^d}{\min}\left\{\<z,y-x> + \frac{1}{2\gamma}\|y-x\|^2 + \Psi(y)\right\}.$
Let $y^k_+ \coloneqq \prox_{\Psi}^{\gamma}\left(\bar x^k - \gamma \bar z^k\right)$. Then, the proximal gradient mapping of $\bar z^k$ at $\bar x^k$ is $\cG(\bar x^k, \bar z^k, \gamma) = \frac{1}{\gamma}(\bar x^k - y^k_+). $
Since $y^{k}_{+}$ is the minimizer of a $1/\gamma$-strongly convex function, we have
\begin{equation*}
\begin{split}
    \<\bar z^k, y^k_+ - \bar x^k> + &\frac{1}{2\gamma}\|y^k_+ - \bar x^k\|^2 +\Psi(y^k_+) \leq \Psi(\bar x^k)  - \frac{1}{2\gamma}\|y^k_+ - \bar x^k\|^2,
\end{split}
\end{equation*}
implying the relation between $\Psi(\bar x^k) - \eta(\bar x^k, \bar z^k)$ and primal convergence: $$\Psi(\bar x^k) - \eta(\bar x^k, \bar z^k) \geq \frac{\gamma}{2} \norm{\cG(\bar x^k, \bar z^k, \gamma)}^2.$$

Following standard practices in optimization,  we set $\gamma = \frac{1}{L_{\nabla F}}$ below for simplicity. However, our algorithms do not require any restriction on the choice of $\gamma$.

\textbf{Step 1:} Leveraging the merit function with $\lambda \asymp \gamma$, we can first obtain an essential lemma (Lemma 11 in Appendix) in our analyses, which says that for sequences $\{x_{i}^k, z_i^k\}_{1\leq i\leq n, k\geq 0}$ generated by \texttt{Prox-DASA(-GT)} (Algorithm \ref{algo: Prox-DASA} or \ref{algo: Prox-DASA-GT}) with $\alpha_k\lesssim \min\{1, (1+\gamma)^{-2}, \gamma^2(1+\gamma)^{-2}\}$, we have
\begin{equation*}
\begin{split}
     W(\bar x^{k+1}, \bar z^{k+1}) - &W(\bar x^{k}, \bar z^{k}) \leq - \alpha_k \left\{\Theta^k  + \Upsilon^k + \alpha_k \Lambda^k + r^{k+1}\right\},
\end{split}
\end{equation*}
where $\E[r^{k+1}\mid\setF_k] = 0$, $\Lambda^k \asymp \gamma\norm{\bar \Delta^{k+1}}^2$, 
\begin{equation*}
\begin{split}
    &\Theta^k \asymp \frac{1}{\gamma} \|\bar x^k  - \bar y^k\|^2 + \gamma \norm{\nabla F(\bar x^k) - \bar z^k}^2, \\
    &\Upsilon^k \asymp \frac{\gamma}{n}\norm{\mZ_k - \bar \mZ_k}^2 + \frac{1}{n\gamma}\norm{\mX_k - \bar \mX_k}^2,
\end{split}
\end{equation*}
and $\bar\Delta^{k+1} = \bar v^{k+1} - \frac{1}{n}\sum_{i=1}^{n} \nabla F_i(x^k_i) = \bar u^{k+1} - \frac{1}{n}\sum_{i=1}^{n} \nabla F_i(x^k_i)$ (for \texttt{Prox-DASA-GT}). Thus, by telescoping and taking expectation with respect to $\setF_0$, we have
\begin{equation}\label{ineq: sketch-main-ineq}
\begin{split}
    &\sum_{k=0}^{K} \alpha_k\E\left[\norm{\bar x^k - \bar y^k}^2 + \gamma^2\norm{\nabla F(\bar x^k) - \bar z^k}^2\right]\\
    &\lesssim  \gamma W(\bar x^0, \bar z^0)+ \gamma^2\sigma^2 \boxed{\sum_{k=0}^{K}\frac{\alpha_k^2}{n}} + \sum_{k=0}^{K} \frac{\alpha_k\left\{\E\left[\norm{\mX_k - \bar \mX_k}^2 + \gamma^2\norm{\mZ_k - \bar \mZ_k}^2\right]\right\}}{n}.
\end{split}
\end{equation}

\textbf{Step 2:} We then analyze the consensus errors. Without loss of generality, we consider $\mX_0 = \bar \mX_0 = \mathbf{0}$, i.e., all nodes have the same initialization at $\mathbf{0}$. For $m\in\mathbb{N}_+$, define
$$\varrho(m) = \frac{(1+\rho^{2m})\rho^{2m}}{(1-\rho^{2m})^2}.$$ Then, we have the following fact:
\begin{itemize}
    \item $\varrho(m)$ is monotonically decreasing with the maximum value being $\varrho(1) = \frac{(1+\rho^{2})\rho^{2}}{(1-\rho^{2})^2}:= \varrho_1$;
    \item $\varrho(m)\leq 1$ if and only if $\rho^{2m}\leq \frac{1}{3}$.
\end{itemize}
With the definition of $\varrho(m)$ and assuming $0<\alpha_{k+1}\leq \alpha_k\leq 1$, we can show the consensus errors have the following upper bounds.

\texttt{Prox-DASA}: Let $\alpha_k \lesssim \varrho(m)^{-\frac{1}{2}}$, we have
\begin{align}
\sum_{k=0}^{K}\frac{\alpha_k}{n} \E\left[\norm{\mX_{k} - \bar \mX_k}^2 \right] \leq \sum_{k=0}^{K}\frac{\gamma^2\alpha_k}{n} \E\left[\norm{\mZ_{k} - \bar \mZ_k}^2\right] \lesssim (\gamma^2\sigma^2 + \nu^2)\varrho(m)\boxed{\sum_{k=0}^{K}\alpha_k^{3}}.\label{ineq: prox-dasa-consensus}
\end{align}

\texttt{Prox-DASA-GT}: Let $\alpha_k \lesssim \min\{\varrho(m)^{-1}, \varrho(m)^{-\frac{1}{2}}\}$, we have
\begin{align}
    \sum_{k=0}^{K}\frac{\alpha_k}{n}\E\left[\|\mX_k - \bar \mX_k\|^2\right]\leq \sum_{k=0}^{K}\frac{\gamma^2\alpha_k}{n}\E\left[\|\mZ_k - \bar \mZ_k\|^2\right] \lesssim  \varrho(m)^2 \boxed{\sum_{k=0}^{K}\alpha_k^{3}}\left\{\gamma^2\sigma^2 + \alpha_k^2\E\left[\|\bar x^k - \bar y^k\|^2\right] \right\}.\label{ineq: sketch-prox-dasa-gt-consensus}
\end{align}
We can also see that to obtain a topology-independent iteration complexity, the number of communication rounds can be set as $m = \lceil \frac{\log 3}{2(1-\rho)} \rceil$, which implies $\varrho(m) \leq 1$.

In addition, we have the following fact that relates the consensus error of $\mY$ to the consensus errors of $\mX$ and $\mZ$:
\begin{equation*}
\begin{split}
        \norm{y^k_+ - \bar y^k}^2 + \frac{1}{n} \norm{\mY_k - \bar \mY_k}^2 = \frac{1}{n} \sum_{i=1}^{n} \norm{y_i^k - y^k_+}^2  \leq \frac{2}{n} \left\{\|\mX_{k} - \bar \mX_k\|^2 + \gamma^2\|\mZ_{k} - \bar \mZ_k\|^2 \right\}.
\end{split}
\end{equation*}

\textbf{Step 3:} Let $R$ be a random integer with 
\[
\Pr(R = k)=\frac{\alpha_k}{\sum_{k=1}^{K}\alpha_k}, \quad k=1,2,\dots, K,
\]
and dividing both sides of \eqref{ineq: prox-dasa-consensus} by $\sum_{k=1}^{K}\alpha_k$, we can obtain that for \texttt{Prox-DASA}, the consensus error of $\mX_R$ satisfies
\begin{equation*}
    \E\left[\frac{1}{n}\norm{\mX_R - \bar \mX_R}^2\right] \lesssim (\gamma^2\sigma^2 +\nu^2)\varrho(m)\frac{\sum_{k=0}^{K} \alpha_k^3}{\sum_{k=1}^{K} \alpha_k}.
\end{equation*}
Moreover, noting that
\begin{equation*}
\begin{split}
    \norm{\cG(\bar x, \nabla F(\bar x), \gamma)}^2 \lesssim \frac{1}{\gamma^2} \left\{\norm{\bar x^k - \bar y^k}^2 
    + \norm{y^k_+ - \bar y^k}^2\right\}  +\norm{\nabla F(\bar x^k) - \bar z^k}^2,
\end{split}
\end{equation*}
and combining \eqref{ineq: sketch-main-ineq} with \eqref{ineq: prox-dasa-consensus}, we can get
\begin{equation*}
    \begin{split}
    &\E\left[\norm{\cG(\bar x^R, \nabla F(\bar x^R), \gamma)}^2\right] \lesssim \underbrace{\boxed{\frac{W(\bar x^0, \bar z^0)}{\gamma\sum_{k=1}^{K}\alpha_k}}}_{\text{initialization-related term}} + \underbrace{\boxed{\sigma^2 \frac{\sum_{k=0}^{K}\alpha_k^2}{n\sum_{k=1}^{K}\alpha_k}}}_{\text{variance-related term}} + \underbrace{\boxed{(\sigma^2 +\gamma^{-2}\nu^2)\varrho(m)\frac{\sum_{k=0}^{K} \alpha_k^3}{\sum_{k=1}^{K} \alpha_k}}}_{\text{consensus error}}.
    \end{split}
\end{equation*}
Thus, setting $\alpha_k \asymp \sqrt{\frac{n}{K}} $, we obtain the convergence results of \texttt{Prox-DASA}:

\begin{equation*}
    \begin{split}
    &\E\left[\norm{\cG(\bar x^R, \nabla F(\bar x^R), \gamma)}^2\right] \lesssim \frac{\gamma^{-1} W(\bar x^0 , \bar z^0) + \sigma^2}{\sqrt{nK}} + \frac{n(\sigma^2 +\gamma^{-2}\nu^2)\varrho(m)}{K},\\
    &\E\left[\frac{1}{\gamma^2 n}\norm{\mX_R - \bar \mX_R}^2\right] \lesssim \frac{n(\sigma^2 +\gamma^{-2}\nu^2)\varrho(m)}{K}.
    \end{split}
\end{equation*}

For \texttt{Prox-DASA-GT}, we can complete the proof with similar arguments by combining \eqref{ineq: sketch-prox-dasa-gt-consensus} with \eqref{ineq: sketch-main-ineq} and noting that $\varrho(m)^2 \alpha_k^4 \lesssim 1$.

\section{Experiments}

\subsection{Synthetic Data}

To demonstrate the effectiveness of our algorithms, we first evaluate our algorithms using synthetic data for solving sparse single index models \citep{alquier2013sparse} in the decentralized setting. We consider the homogeneous setting where the data sample at each node $\xi=(X, Y)$ is generated from the same single index model $Y = g(X^\top \theta_*) + \varepsilon$, where $X, \theta \in \realset^d$ and $\E[ \varepsilon| X]=0$. In this case, we solve the following $L_1$-regularized least square problems:
\begin{equation*}
    \underset{\theta\in \realset^d}{\min}~ \frac{1}{n}\sum_{i=1}^{n} \underset{(X, Y)\sim \cD}{\E}\left[(Y - g(X^\top\theta))^2\right] + \lambda \norm{\theta}_1
\end{equation*}
In particular, we set $\theta_*\in\realset^{100}$ to be a sparse vector and $g(\cdot) = (\cdot)^2$ which corresponds to the sparse phase retrieval problem \citep{jaganathan2016phase}. We simulate streaming data samples with batch size $=1$ for training and 10,000 data samples per node for evaluations, where $X$ and $\epsilon$ are sampled independently from two Gaussian distributions. We employ a ring topology for the network where self-weighting and neighbor weights are set to be $1/3$. We set the penalty parameter $\lambda = 0.01$, the total number of iterations $K=10,000$, $\alpha_k = \sqrt{n/K}$, $\gamma=0.01$, and the number of communication rounds per iteration $m=\lceil \frac{1}{1-\rho} \rceil$. We plot the test loss and the norm of proximal gradient mapping in the log scale against the number of iterations in Figure \ref{fig:linear-speedup}, which shows that our decentralized algorithms have an additional linear speed-up with respect to $n$. In other words, the algorithms become faster as more agents are added to the network.
\begin{figure}[!t]
    \centering
    \includegraphics[width=.7\textwidth]{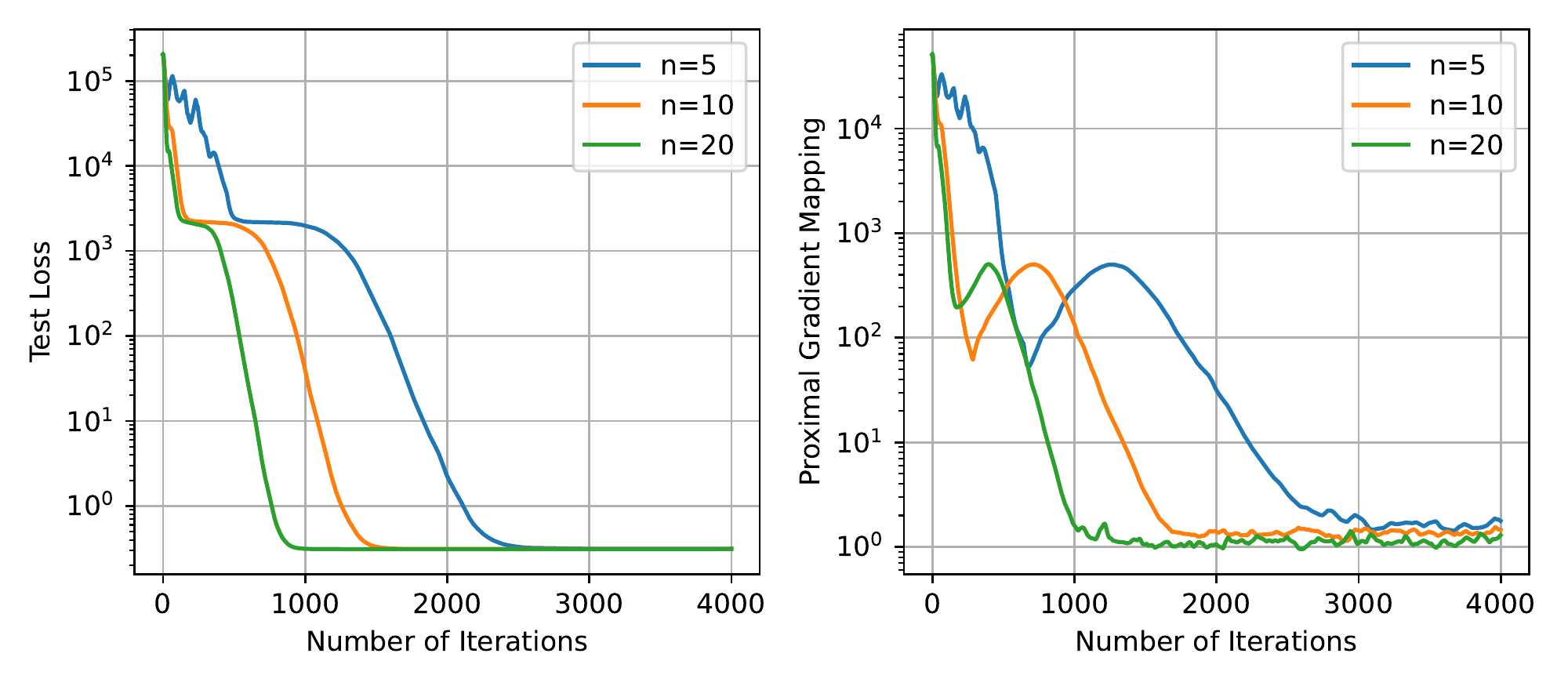}
    \caption{Linear-speedup performance of \texttt{Prox-DASA} for decentralized online sparse phase retrievel problems. (\texttt{Prox-DASA-GT} has relatively the same plots)}
    \label{fig:linear-speedup}
\end{figure}

\subsection{Real-World Data}\label{sec: real_data_exp}
\begin{figure*}
    \centering
    \subfigure[]{\label{a9a_acc_time}\includegraphics[width=0.32\textwidth]{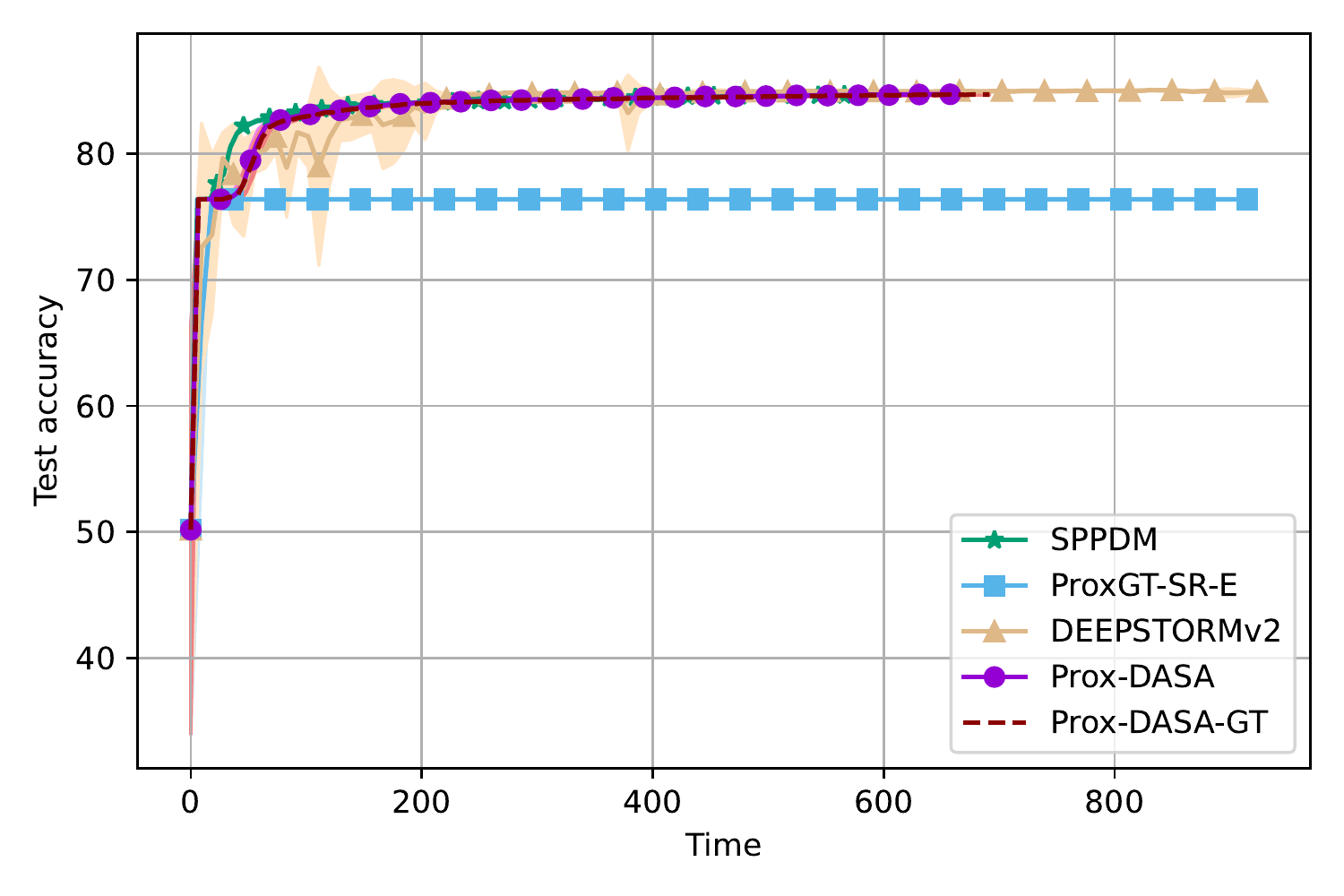}}
    \subfigure[]{\label{a9a_loss_time}\includegraphics[width=0.32\textwidth]{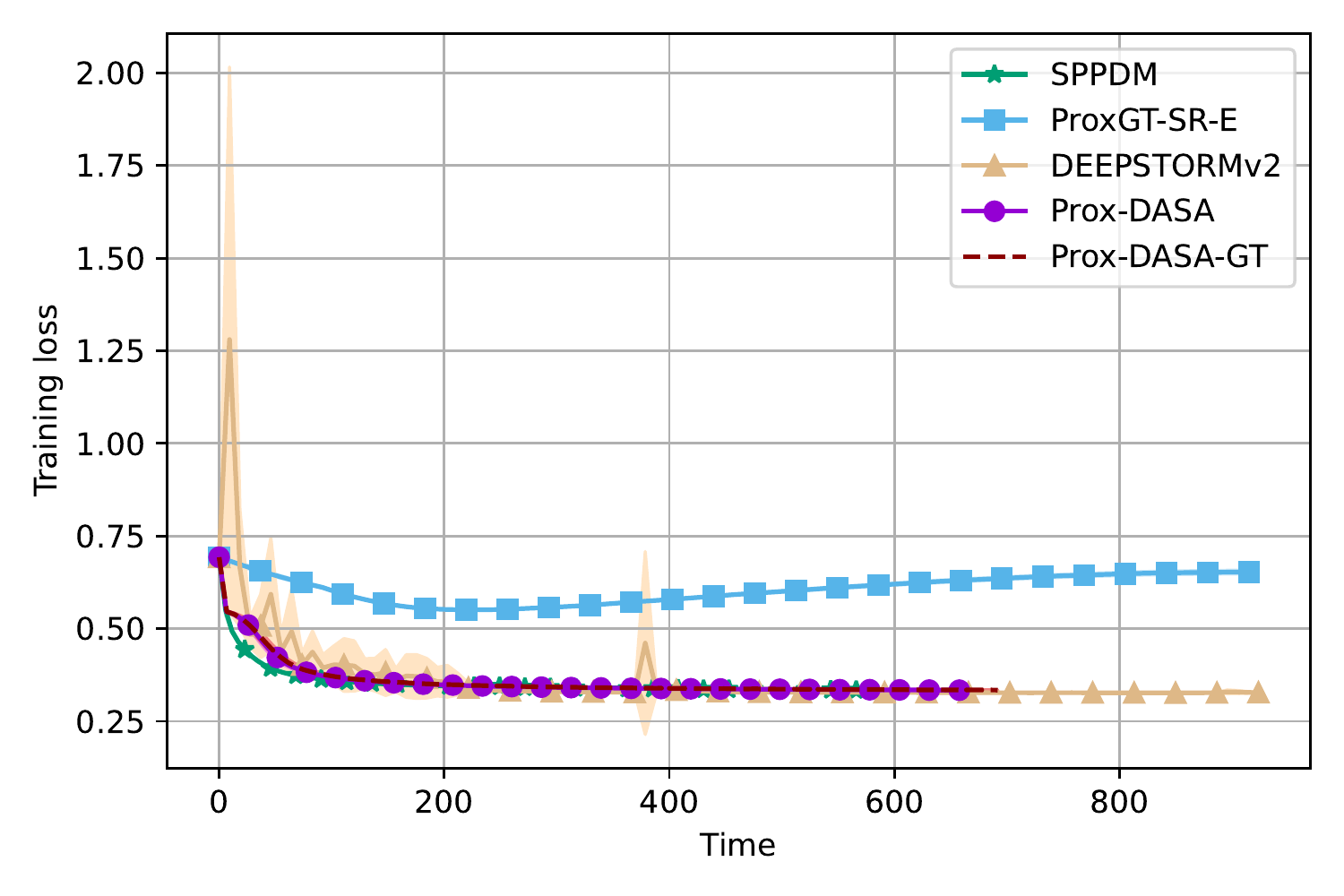}}
    \subfigure[]{\label{a9a_stat_time}\includegraphics[width=0.32\textwidth]{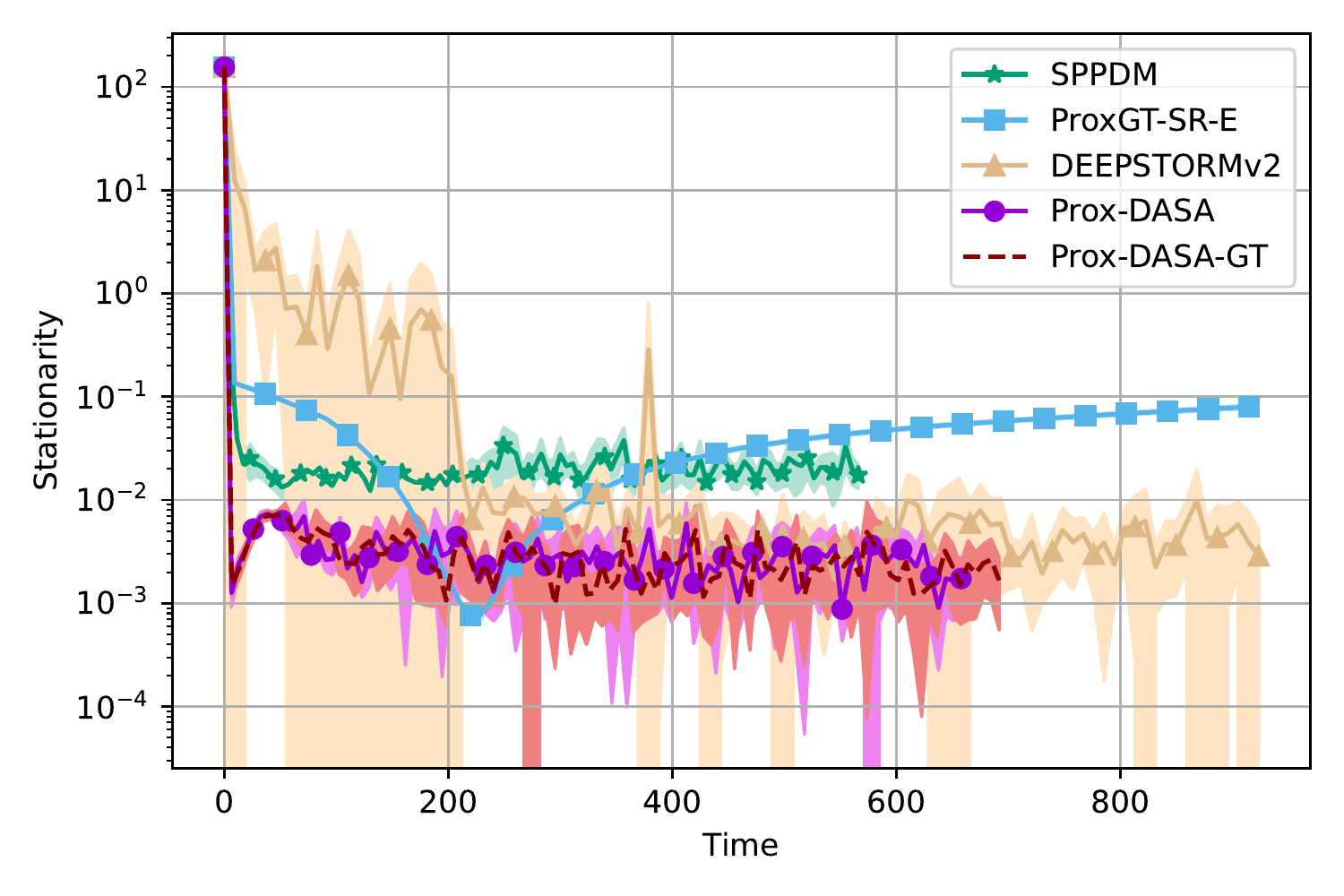}}
    
    \subfigure[]{\label{a9a_acc_epo}\includegraphics[width=0.32\textwidth]{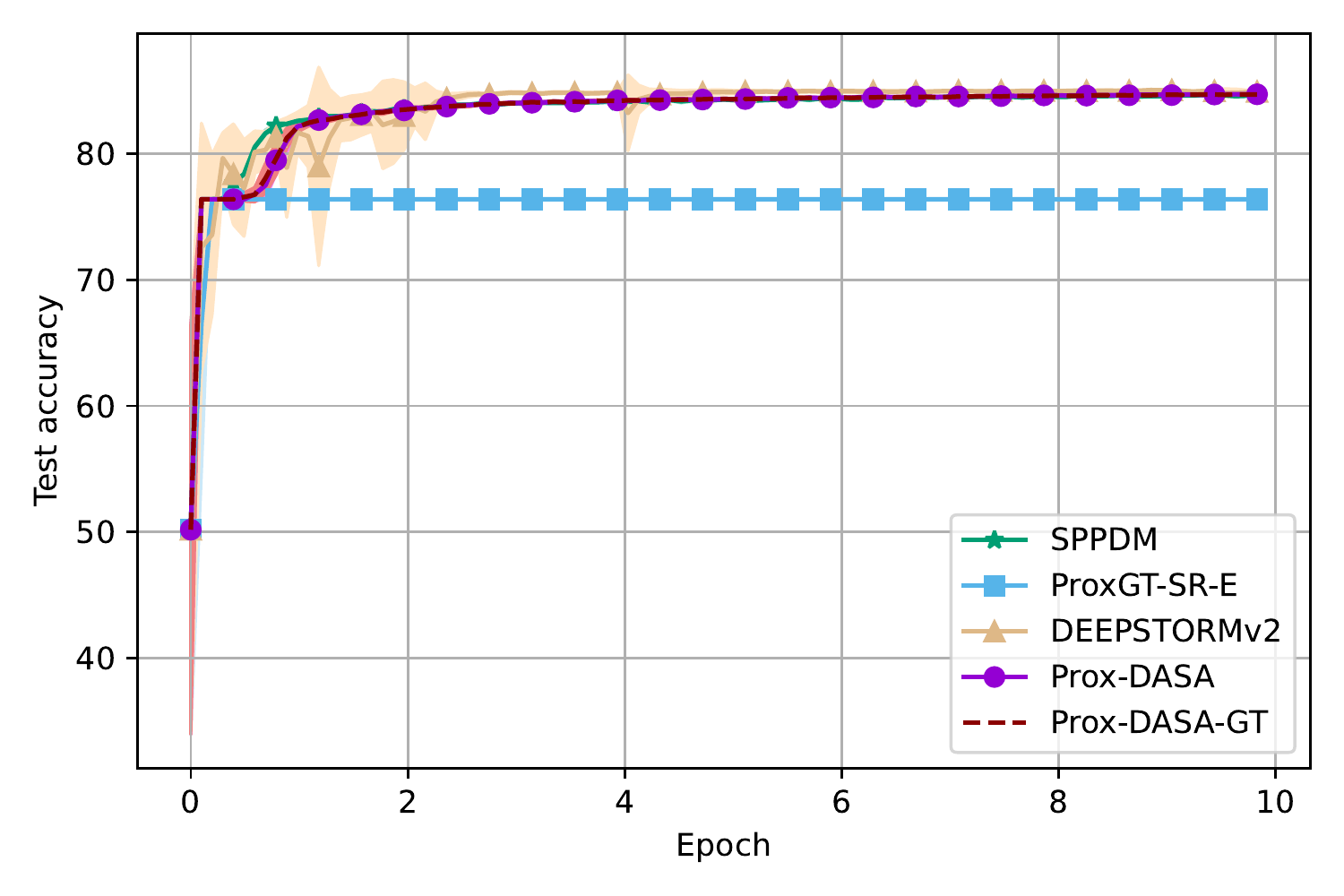}}
    \subfigure[]{\label{a9a_loss_epo}\includegraphics[width=0.32\textwidth]{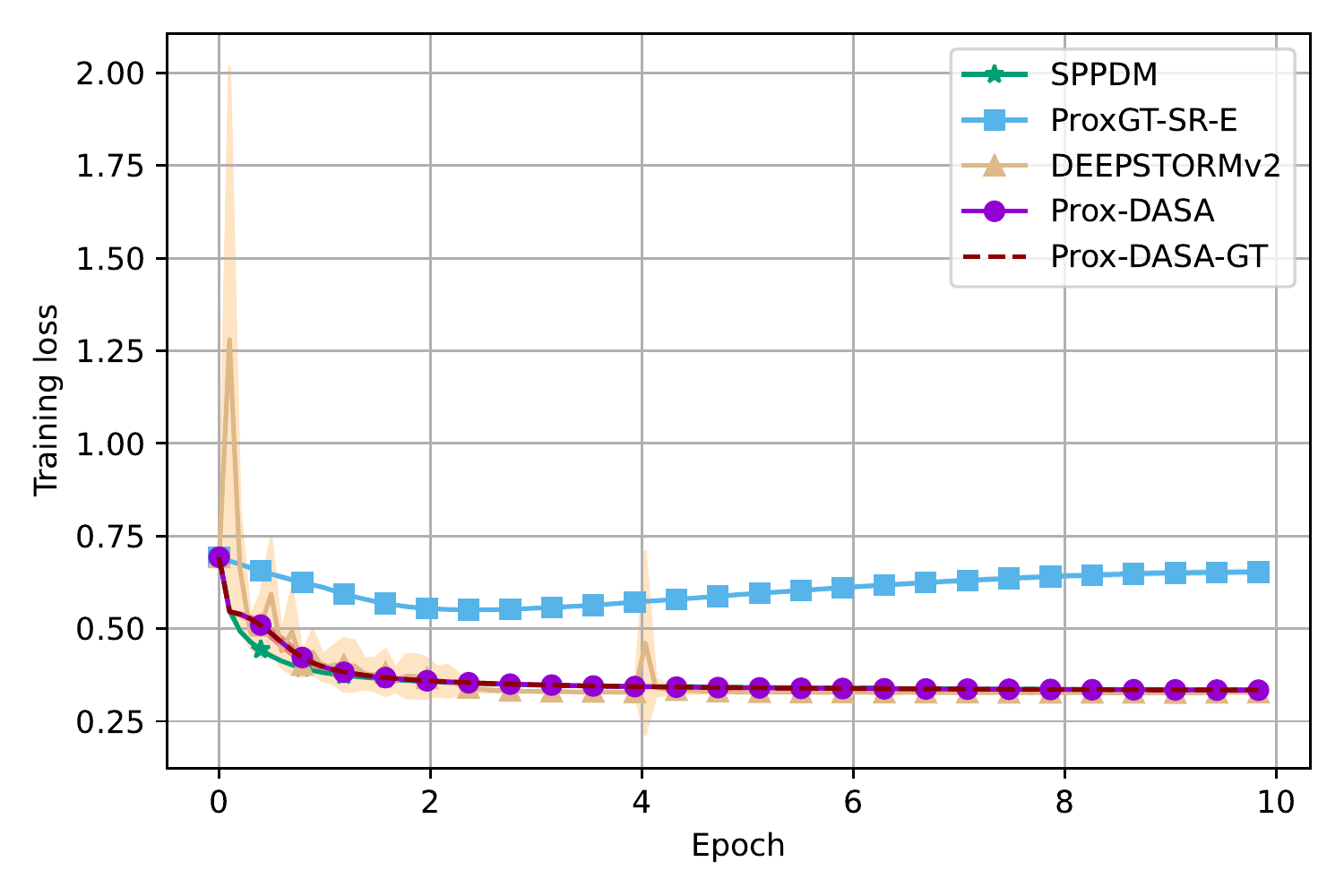}}
    \subfigure[]{\label{a9a_stat_epo}\includegraphics[width=0.32\textwidth]{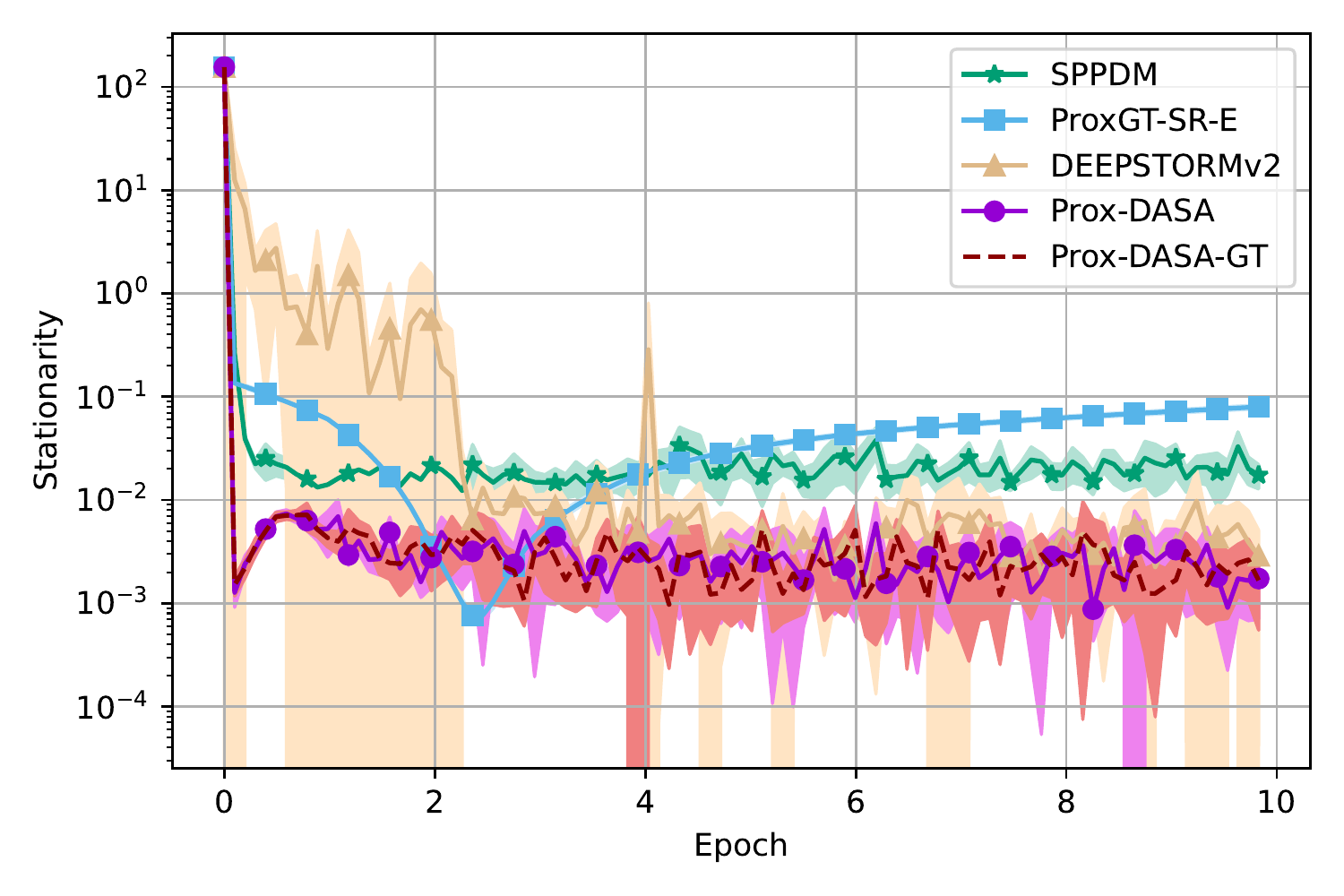}}

    \subfigure[]{\label{mnist_acc_time}\includegraphics[width=0.32\textwidth]{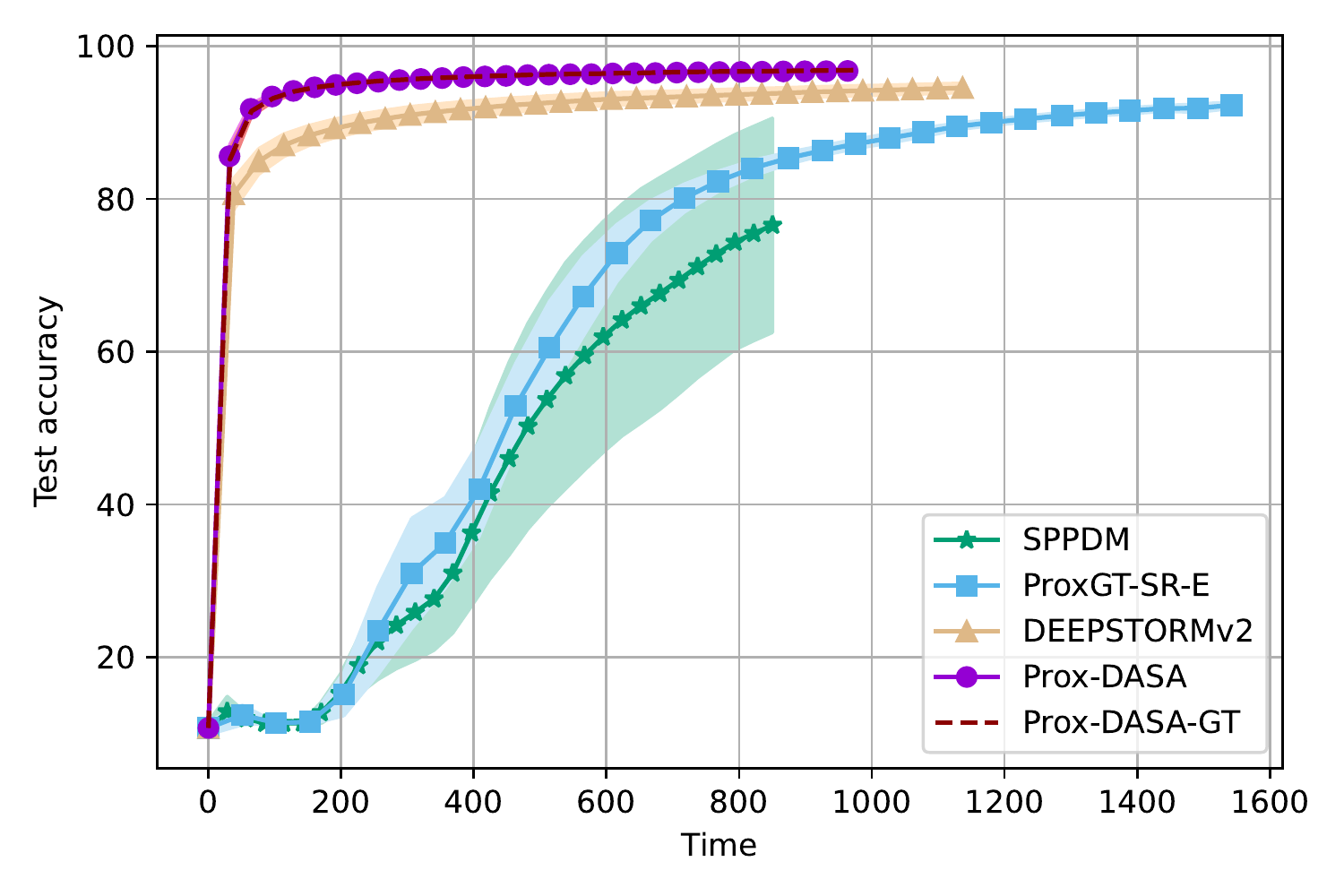}}
    \subfigure[]{\label{mnist_loss_time}\includegraphics[width=0.32\textwidth]{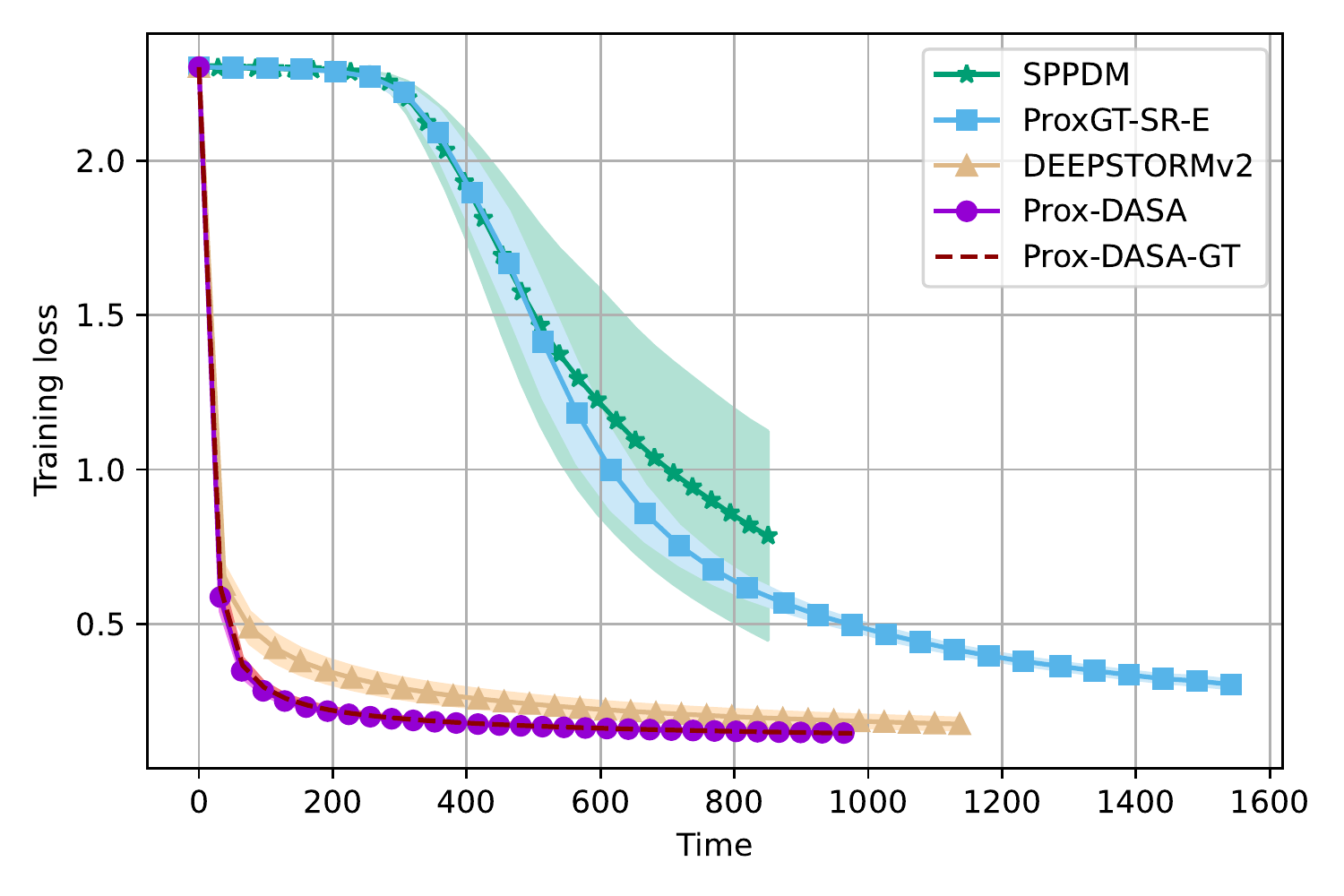}}
    \subfigure[]{\label{mnist_stat_time}\includegraphics[width=0.32\textwidth]{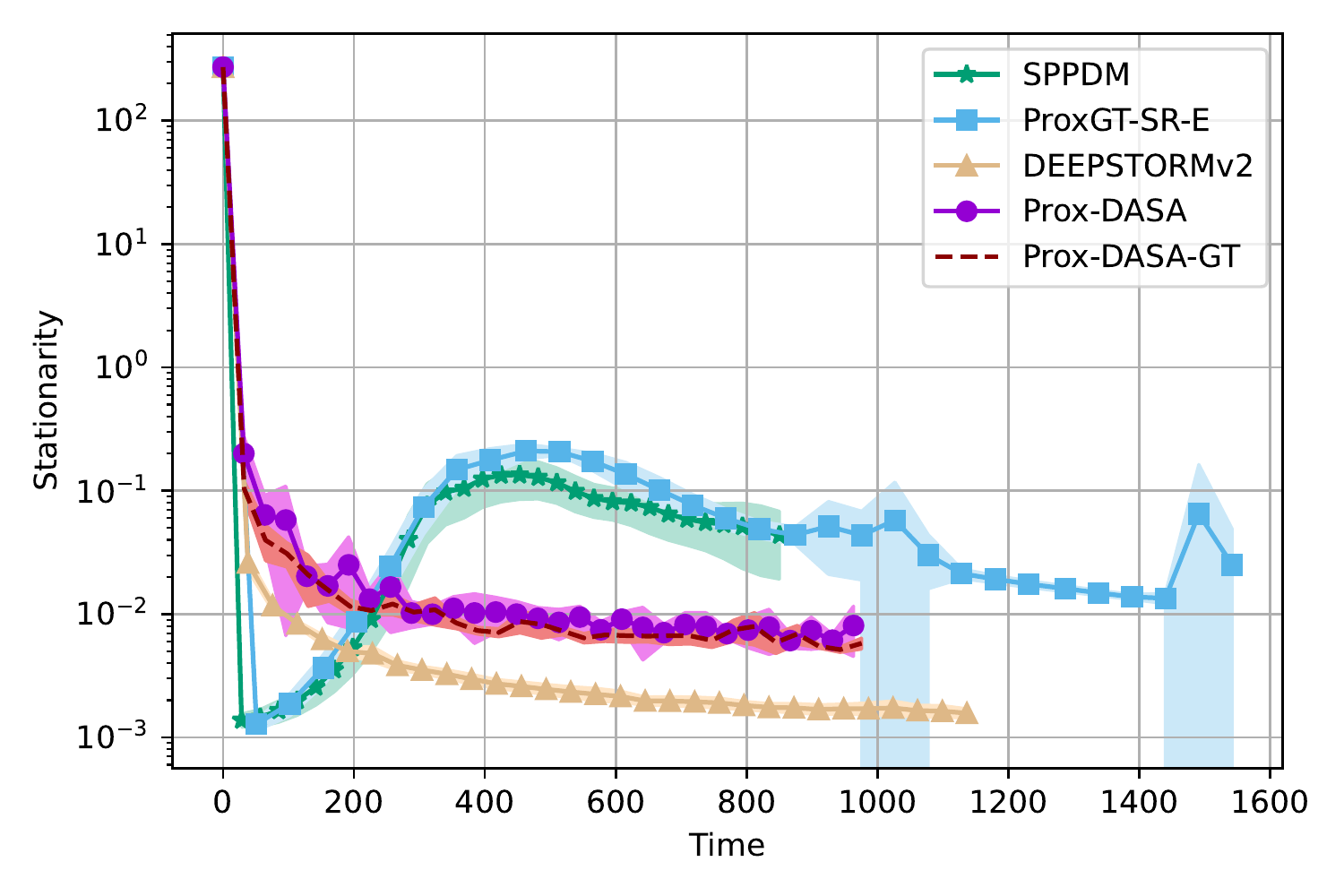}}
    
    \subfigure[]{\label{mnist_acc_epo}\includegraphics[width=0.32\textwidth]{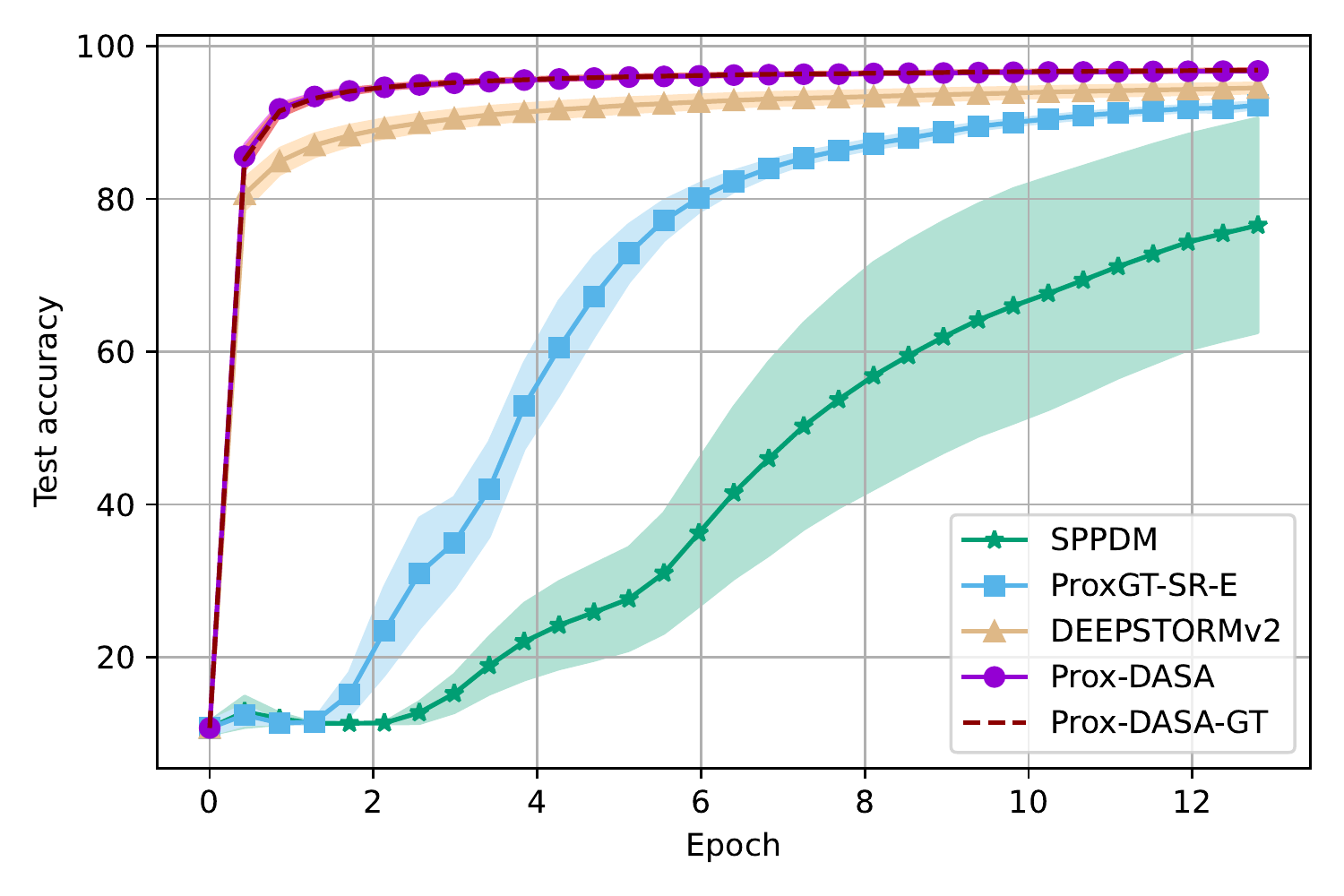}}
    \subfigure[]{\label{mnist_loss_epo}\includegraphics[width=0.32\textwidth]{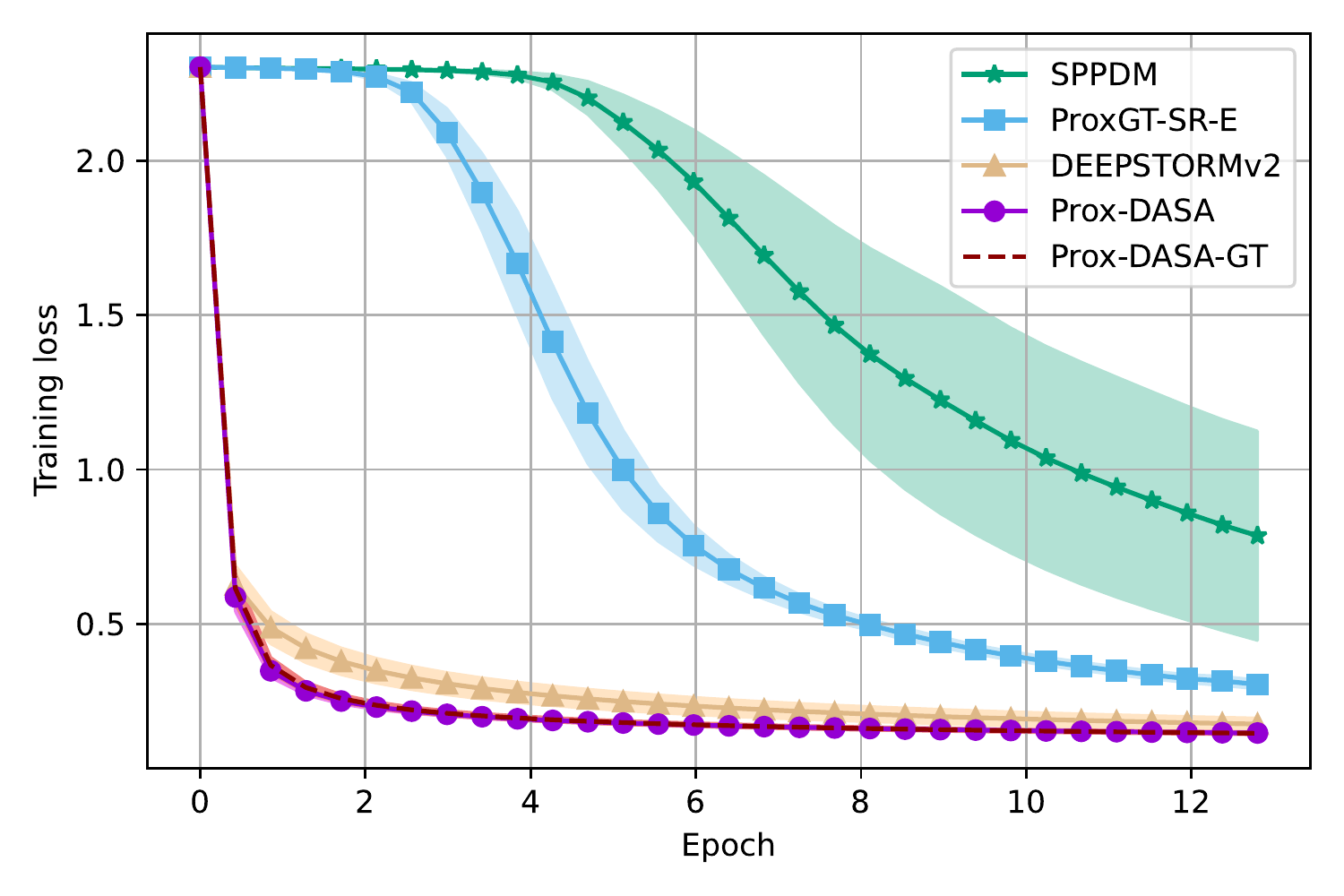}}
    \subfigure[]{\label{mnist_stat_epo}\includegraphics[width=0.32\textwidth]{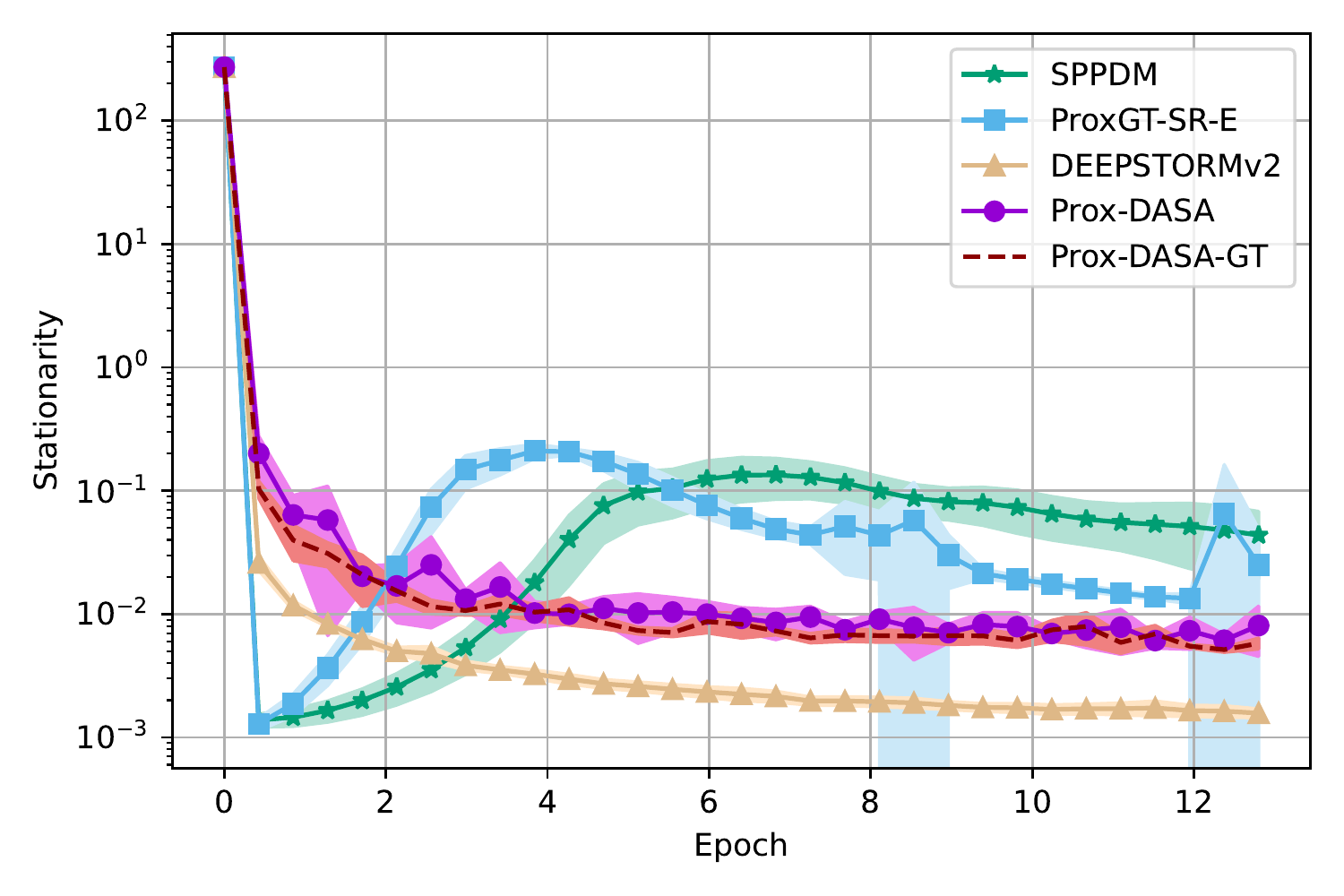}}
    \caption{Comparisons between \texttt{SPPDM} \citep{wang2021distributed}, \texttt{ProxGT-SR-E} \citep{xin2021stochastic}, \texttt{DEEPSTORM} \citep{mancino2022proximal}, \texttt{Prox-DASA} \ref{algo: Prox-DASA}, and \texttt{Prox-DASA-GT} \ref{algo: Prox-DASA-GT}. The first two rows correspond to a9a and the last two rows correspond to MNIST. The results are averaged over 10 trials, and the shaded regions represent confidence intervals. The vertical axes in the third column are log-scale. It should be noted that \texttt{ProxGT-SR-E} maintains another hyperparameter $q$ (see, e.g., Algorithm 4 and Theorem 3 in \citep{xin2021stochastic}) and computes gradients using a full batch every $q$ iterations. For simplicity, we do not include that amount of epochs when we plot this figure. In other words, the real number of epochs required to obtain a point on \texttt{ProxGT-SR} is larger than plotted in the figures in the second and fourth rows. We include the plots that take $q$ into account in Appendix.}\label{fig: comparison} 
\end{figure*}
Following \cite{mancino2022proximal}, we consider solving the classification problem 
\begin{equation}\label{eq:exp_opt}
    \underset{\theta\in\realset^d}{\min}~ \frac{1}{n} \sum_{i=1}^{n} \frac{1}{|\mathcal{D}_i|}\sum_{(x, y)\in \mathcal{D}_i}\ell_i(f(x;\theta),y) + \lambda \|\theta\|_1,
\end{equation}
on a9a and MNIST datasets\footnote{Available at https://www.openml.org.}. Here, $\ell_i$ denotes the cross-entropy loss, and $f$ represents a neural network parameterized by $\theta$ with $x$ being its input. $\mathcal{D}_i$ is the training set only available to agent $i$. The $L_1$ regularization term is used to impose a sparsity structure on the neural network. We use the code in \cite{mancino2022proximal} for \texttt{SPPDM}, \texttt{ProxGT-SR-O/E}, \texttt{DEEPSTORM}, and then implement \texttt{Prox-DASA} and \texttt{Prox-DASA-GT} under their framework, which mainly utilizes PyTorch \citep{NEURIPS2019_9015} and mpi4py \citep{dalcin2021mpi4py}.  We use a 2-layer perception model on a9a and the LeNet architecture \citep{lecun2015lenet} for the MNIST dataset. We have 8 agents ($n=8$) which connect in the form of a ring for a9a and a random graph for MNIST. To demonstrate the performance of our algorithms in the constant batch size setting, the batch size is chosen to be 4 for a9a and 32 for MNIST for all algorithms. The learning rates provided in the code of \cite{mancino2022proximal} are adjusted accordingly, and we select the ones with the best performance. For \texttt{Prox-DASA} and \texttt{Prox-DASA-GT} we choose a diminishing stepsize sequence, namely, $\alpha_k =  \min\left\{\alpha\sqrt{\frac{n}{k}}, 1\right\}$ for all $k\geq 0$. Note that the same complexity (up to logarithmic factors) bounds can be obtained by directly plugging in the aforementioned expressions for $\alpha_k$ in Section \ref{sec: proof_sketch}. Then we tune $\gamma\in \left\{1, 3, 10\right\}$ and $\alpha\in \left\{0.3, 1.0, 3.0\right\}$. The penalty parameter $\lambda$ is chosen to be $0.0001$ for all experiments. The number of communication rounds per iteration $m$ is set to be $1$ for all algorithms. We evaluate the model performance periodically during training and then plot the results in Figure \ref{fig: comparison}, from which we observe that both \texttt{Prox-DASA} and \texttt{Prox-DASA-GT} have considerably good performance with small variance in terms of test accuracy, training loss, and stationarity. In particular, it should be noted that although \texttt{DEEPSTORM} achieves better stationarity in Figure \ref{mnist_stat_epo} and \ref{mnist_stat_time}, training a neural network by using \texttt{DEEPSTORM} takes longer time than \texttt{Prox-DASA} and \texttt{Prox-DASA-GT} since it uses the momentum-based variance reduction technique, which
requires {\bf two forward-backward passes} (see, e.g., Eq. (10) and Algorithm 1 in \cite{mancino2022proximal}) to compute the gradients in one iteration per agent. In contrast, ours only require {\bf one}, which saves a large amount of time (see Table 1 in Appendix). We include further details of our experiments in the Appendix. 

\section{Conclusion}

In this work, we propose and analyze a class of single time-scale decentralized proximal algorithms (\texttt{Prox-DASA-(GT)}) for non-convex stochastic composite optimization in the form of \eqref{eq:problem}. We show that our algorithms achieve linear speed-up with respect to the number of agents using an $\cO(1)$ batch size per iteration under mild assumptions. Furthermore, we demonstrate the efficiency and effectiveness of our algorithms through extensive experiments, in which our algorithms achieve relatively better results with less training time using a small batch size compared to existing methods. In future research, it would be intriguing to expand our work in the context of dependent and heavy-tailed stochastic gradient scenarios \citep{wai2020convergence, li2022high}.

\section*{Acknowledgements}
    We thank the authors of~\citep{mancino2022proximal} for kindly providing the code framework to support our experiments.
    The research of KB is supported by NSF grant DMS-2053918. The research of SG is partially supported by NSERC grant RGPIN-2021-02644.

\bibliographystyle{plainnat}
\bibliography{bibfile}


\appendix
\section{Experimental Details}
\label{sec: exp_details}

All experiments are conducted on a laptop with Intel Core i7-11370H Processor and Windows 11 operating system. The total iteration numbers for a9a and MNIST are 10000 and 3000 respectively. The graph that represents the network topology is set to be ring (or cycle in graph theory) for a9a and random graph (given by \cite{mancino2022proximal}) for MNIST (See Figure \ref{fig: graphs}). To demonstrate the performance of our algorithms in a constant batch size setting, the batch sizes are chosen to be $4$ for a9a and $32$ for MNIST in all algorithms. We adjust the learning rates provided in the code of \cite{mancino2022proximal} accordingly and select the ones that have the best performance. For \texttt{Prox-DASA} and \texttt{Prox-DASA-GT} we choose a diminishing stepsize sequence, namely, $\alpha_k =  \min\left\{\alpha\sqrt{\frac{n}{k}}, 1\right\}$ for all $k\geq 0$. Note that the same complexity (up to logarithmic factors) bounds can be obtained by directly plugging in the aforementioned expressions for $\alpha_k$ in Section 4.3. Then we tune $\gamma\in \left\{1, 3, 10\right\}$ and $\alpha\in \left\{0.3, 1.0, 3.0\right\}$. The penalty parameter $\lambda$ is chosen to be $0.0001$ for all experiments.

We summarize the outputs of all experiments in Table \ref{table: output}, from which we can tell \texttt{Prox-DASA} and \texttt{Prox-DASA-GT} achieve good performance in a relatively short amount of time. The stationarity is defined as $\|\cG(\bar x^k, \nabla F(\bar x^k), 1)\|^2 + \|\mX_k - \bar \mX_k\|^2$, which is the same as that in \cite{mancino2022proximal}. As mentioned in the caption of Figure 2 in the main paper, there is an extra hyperparameter $q$ in \texttt{ProxGT-SR-E}, and we found that large $q$ already works well for a9a experiment, but $q$ has to be small in the MNIST experiment otherwise the final accuracy will be much smaller than the one presented in Table \ref{table: output}. Hence in \texttt{ProxGT-SR-E} we choose $q=1000$ for a9a and $q=32$ for MNIST, and the plots that take this amount of epochs into account are in Figure \ref{fig: full_appendix}.\\

\begin{table*}[h]
\centering
\renewcommand{\arraystretch}{1.1}
\caption{Comparisons between all algorithms}\label{table: output}
\resizebox{\textwidth}{!}{%
\begin{tabular}{|c|c|c|c|c|c|c|}
\hline
\textbf{Algorithm} & \textbf{Accuracy} & \textbf{Training Loss} & \textbf{Stationarity} & \makecell{\bf Communication \\ \bf time per iteration (s)} & \makecell{\bf Computation \\ \bf time per iteration (s)} & \makecell{\bf Total time \\ \bf per iteration (s)} \\ \hline
\multicolumn{7}{|c|}{\bf a9a} \\ \hline
\texttt{SPPDM}              & 84.64\%             & 0.3340            & 0.0174                & 0.0260                      & 0.0305                    & 0.0565              \\ \hline
\texttt{ProxGT-SR-E}        & 76.38\%             & 0.6528             & 0.0797                & 0.0521                      & 0.0394                    & 0.0915              \\ \hline
\texttt{DEEPSTORM v2}          & \bf 84.90\%             & \bf 0.3274             & 0.0029                & 0.0525                      & 0.0398                    & 0.0923              \\ \hline
\texttt{Prox-DASA}          & 84.71\%             & 0.3338             & \bf 0.0017                & 0.0360                      & 0.0298                    & 0.0658              \\ \hline
\texttt{Prox-DASA-GT}       & 84.69\%             & 0.3342             &  \bf 0.0017                & 0.0390                      & 0.0301                    & 0.0691              \\ \hline
\multicolumn{7}{|c|}{\bf MNIST} \\ \hline
\texttt{SPPDM}              & 76.54\%           & 0.7854                 & 0.0436                & 0.1587                      & 0.1246                    & 0.2833              \\ \hline
\texttt{ProxGT-SR-E}        & 92.26\%           & 0.3042                 & 0.0250                & 0.1771                      & 0.3368                    & 0.5139              \\ \hline
\texttt{DEEPSTORM v2}          & 94.52\%           & 0.1759                 &  \bf 0.0016                & 0.1758                      & 0.2030                    & 0.3788              \\ \hline
\texttt{Prox-DASA}          & 96.74\%           & 0.1469                 & 0.0081                & 0.1912                      & 0.1299                    & 0.3211              \\ \hline
\texttt{Prox-DASA-GT}       & \bf 96.84\%           & \bf 0.1460                 & 0.0058                & 0.1935                      & 0.1317                    & 0.3252              \\ \hline
\end{tabular}
}
\end{table*}

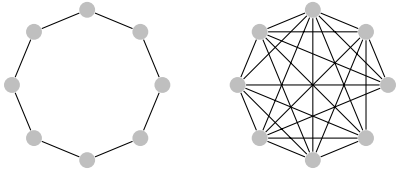
\begin{figure*}[h]
    \centering
    \begin{tikzpicture}[shorten >=1pt,->]
      \tikzstyle{vertex}=[circle,fill=black!25,minimum size=6pt,inner sep=2pt]
      \node[vertex] (G_1) at (0, 1) {};
      \node[vertex] (G_2) at (0.707, 0.707)   {};
      \node[vertex] (G_3) at (1, 0)  {};
      \node[vertex] (G_4) at (0.707, -0.707)  {};
      \node[vertex] (G_5) at (0, -1)  {};
      \node[vertex] (G_6) at (-0.707, -0.707)  {};
      \node[vertex] (G_7) at (-1, 0)  {};
      \node[vertex] (G_8) at (-0.707, 0.707)  {};
      \draw (G_1) -- (G_2) -- (G_3) -- (G_4) -- (G_5) -- (G_6) -- (G_7) -- (G_8) -- (G_1) -- cycle;
      \node[vertex] (v1) at (3, 1) {};
      \node[vertex] (v2) at (3.707, 0.707)   {};
      \node[vertex] (v3) at (4, 0)  {};
      \node[vertex] (v4) at (3.707, -0.707)  {};
      \node[vertex] (v5) at (3, -1)  {};
      \node[vertex] (v6) at (3-0.707, -0.707)  {};
      \node[vertex] (v7) at (3-1, 0)  {};
      \node[vertex] (v8) at (3-0.707, 0.707)  {};
      \draw (v1) -- (v2) -- (v3) -- (v6) -- (v4) -- (v5) -- (v7) -- (v8) -- (v1) --cycle;
      \draw (v1) -- (v3) -- (v7) -- (v4) -- (v8) -- (v2) -- (v5) -- (v6) -- (v1) -- cycle;
      \draw (v1) -- (v4) -- (v2) -- (v6) -- (v7) -- (v1) --cycle;
      \draw (v1) -- (v5) -- (v8) -- cycle;
      \draw (v3) -- (v8) -- cycle;
    \end{tikzpicture}
    \caption{Network topology. The left represents the ring topology and the right represents (an instance of) the random graph.}
    \label{fig: graphs}
\end{figure*}

\begin{figure*}
    \centering
    \subfigure[]{\label{a9a_acc_appendix}\includegraphics[width=0.32\textwidth]{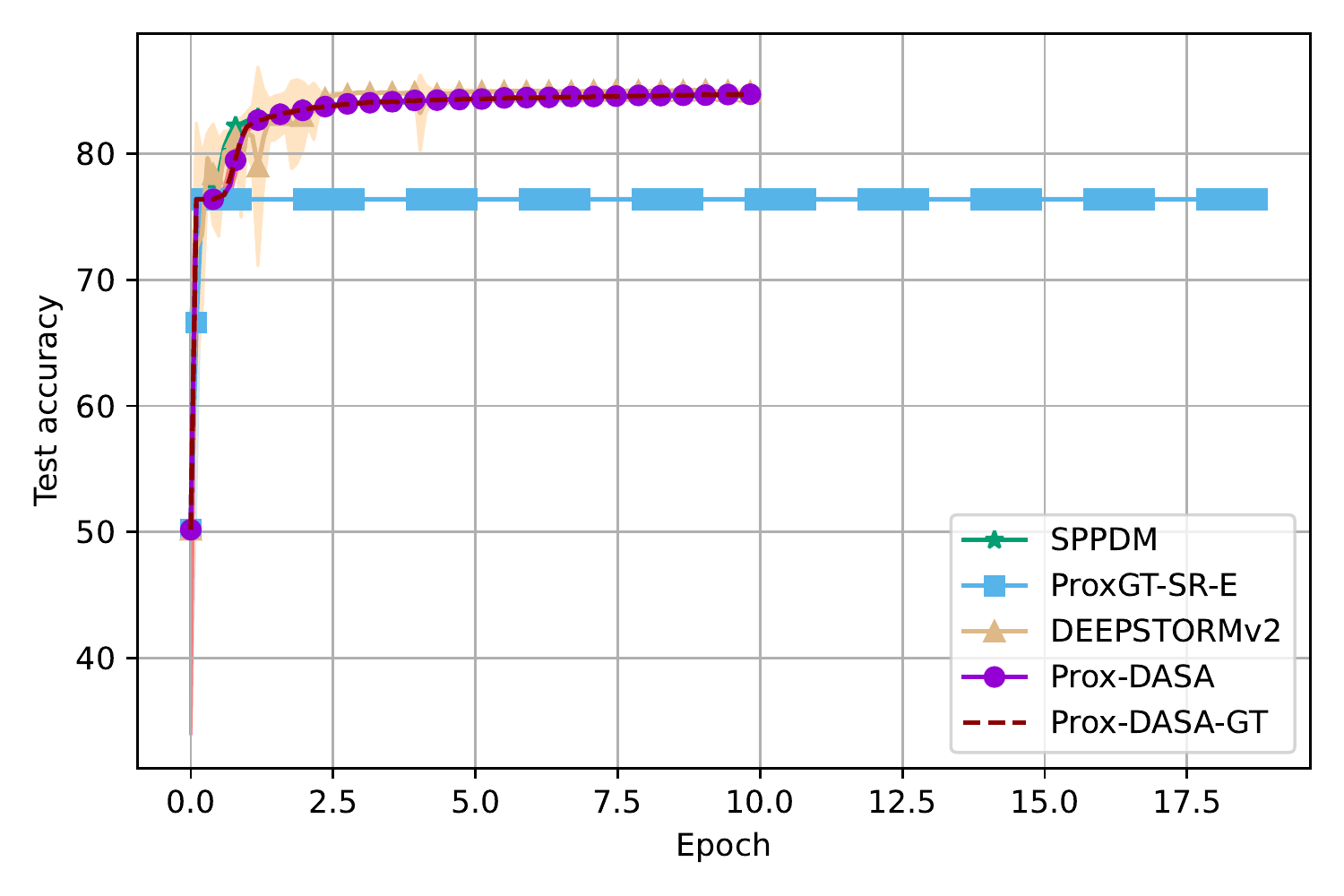}}
    \subfigure[]{\label{a9a_loss_appendix}\includegraphics[width=0.32\textwidth]{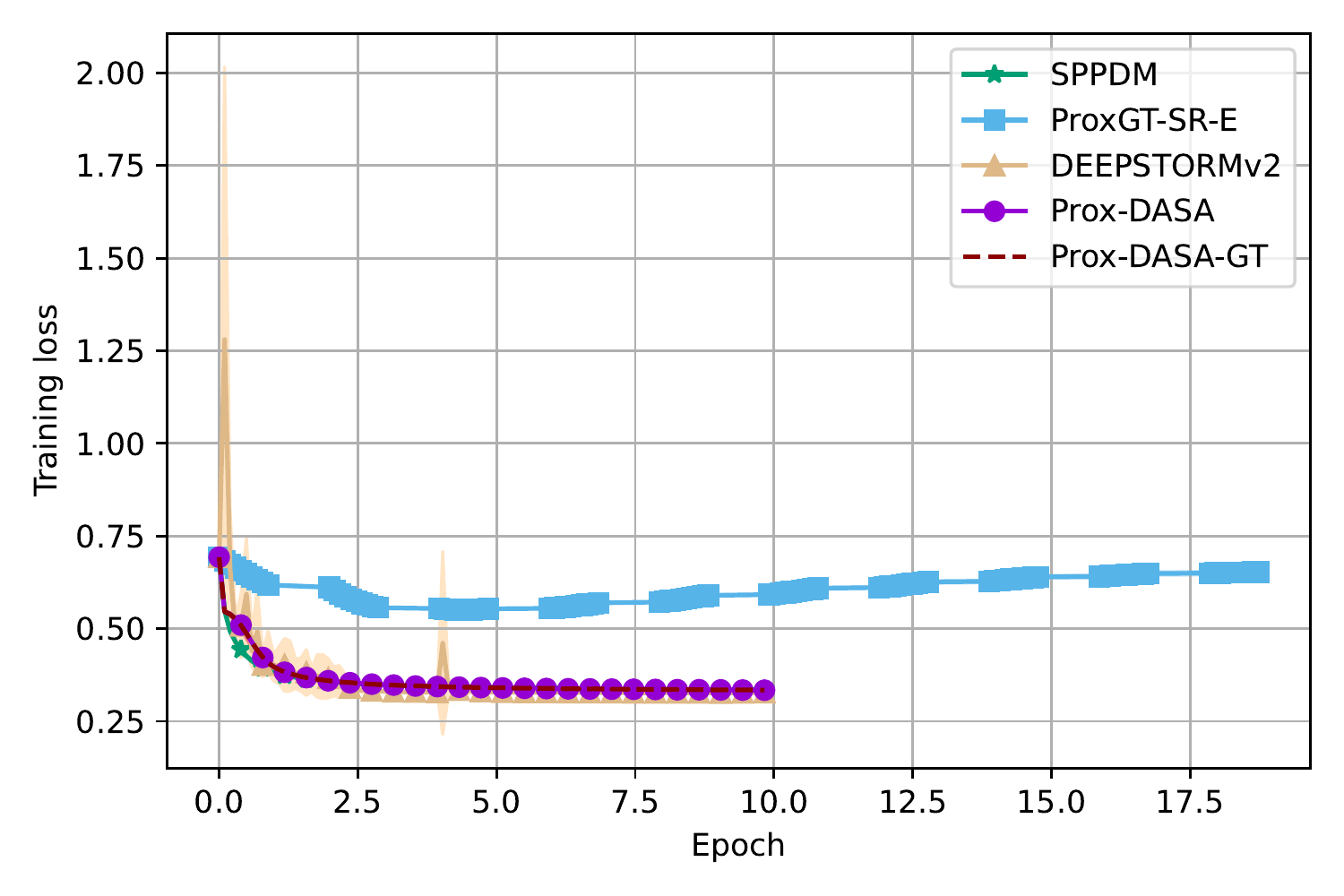}}
    \subfigure[]{\label{a9a_stat_appendix}\includegraphics[width=0.32\textwidth]{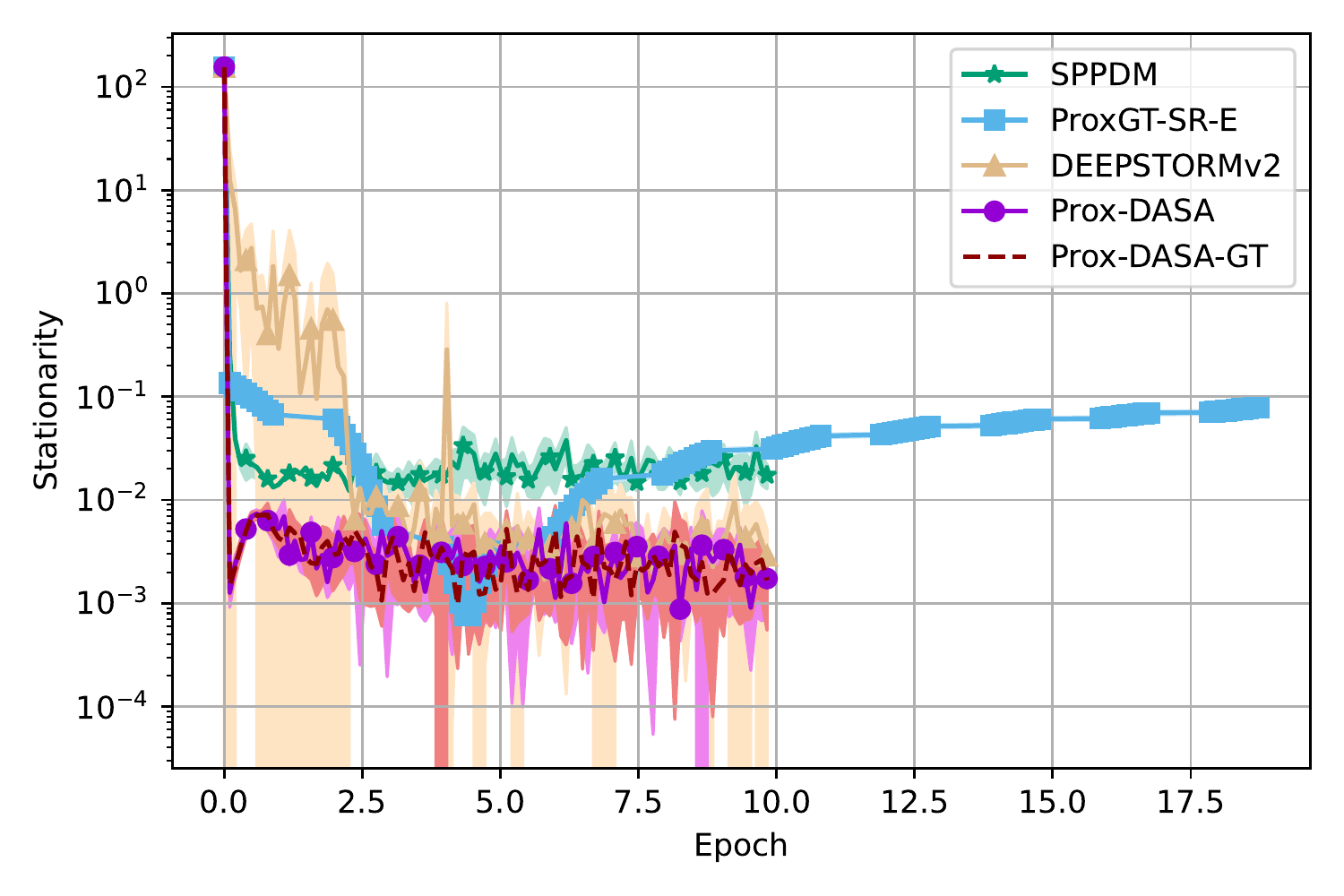}}

    \subfigure[]{\label{mnist_acc_appendix}\includegraphics[width=0.32\textwidth]{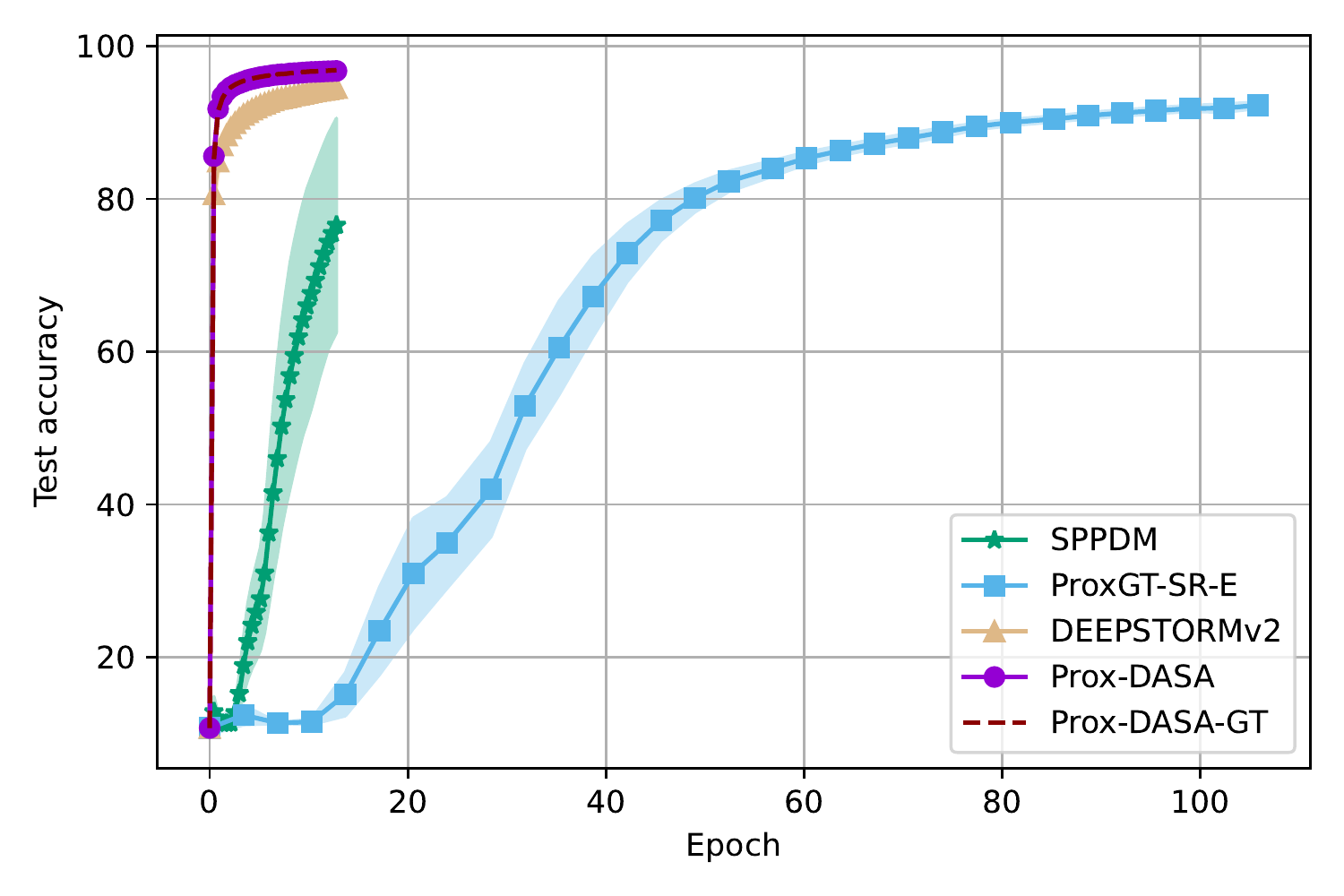}}
    \subfigure[]{\label{mnist_loss_appendix}\includegraphics[width=0.32\textwidth]{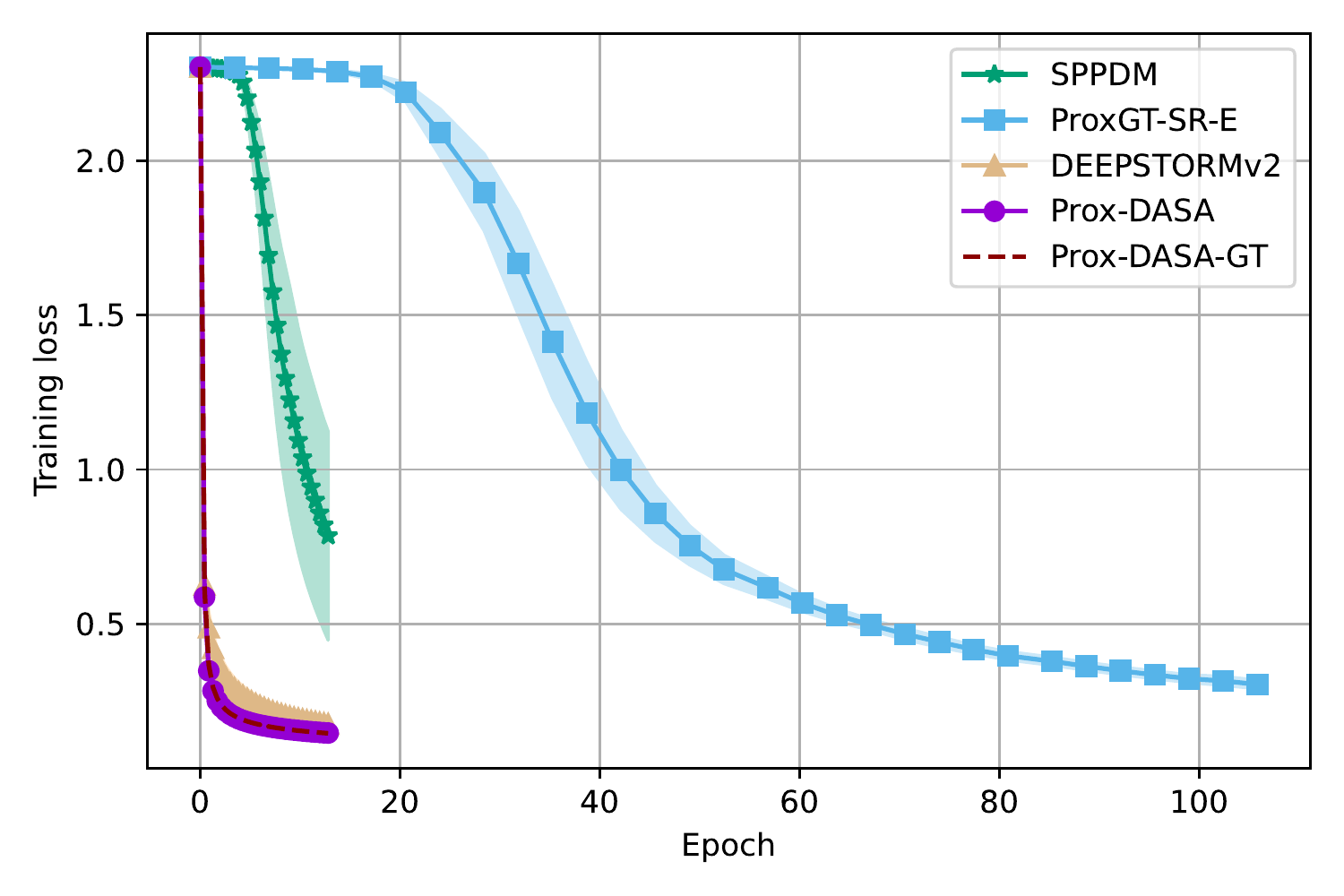}}
    \subfigure[]{\label{mnist_stat_appendix}\includegraphics[width=0.32\textwidth]{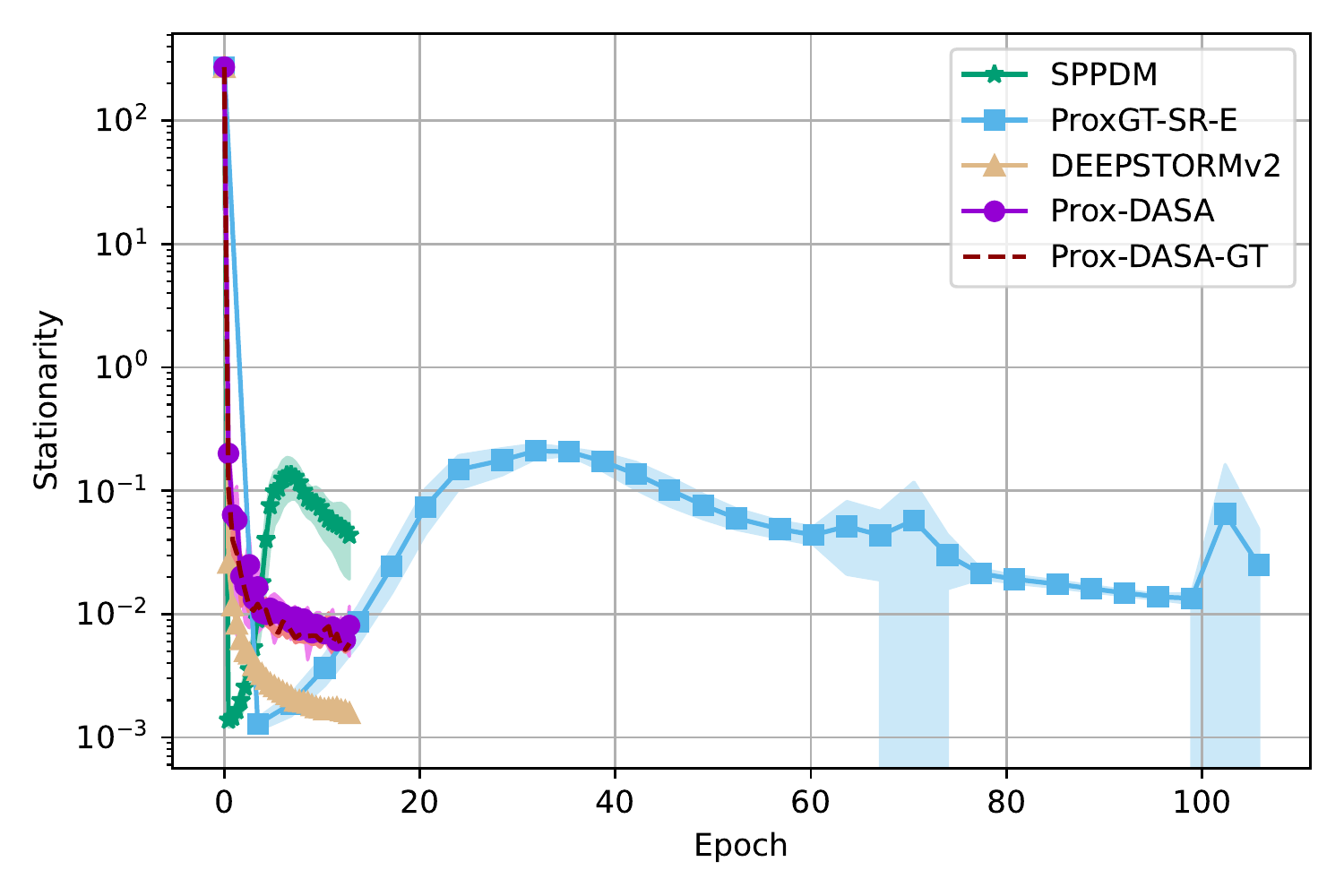}}
    \vspace{-0.2cm}
    \caption{Comparisons between \texttt{SPPDM} \citep{wang2021distributed}, \texttt{ProxGT-SR-E} \citep{xin2021stochastic}, \texttt{DEEPSTORM} \citep{mancino2022proximal}, \texttt{Prox-DASA} \ref{algo: Prox-DASA}, and \texttt{Prox-DASA-GT} \ref{algo: Prox-DASA-GT}. In each experiment \texttt{ProxGT-SR-E} computes 1 more epoch than other algorithms every $q$ iterations. $q$ is chosen to be 1000 for a9a and 32 for MNIST.}\label{fig: full_appendix}
    \vspace{-0.3cm}
\end{figure*}

\section{Accelerated Consensus}
\label{appendix: acc-consensus}
 When the number of communication round $m>1$, we can replace $\mW^m$ with the Chebyshev mixing protocol described in Algorithm \ref{algo: acc-consensus}. Then, we have the following lemma.
\begin{lemma}\label{lem: chebyshev_mixing}
    Suppose $\mW$ satisfies Assumption \ref{aspt:gossipMatrix}. Let $\mA_0, \mA_m$ be the input and output matrix of Algorithm \ref{algo: acc-consensus} respectively. Then, we have
    \[
        \norm{\mA_m - \bar \mA_m} \leq 2\left( 1-\sqrt{1-\rho} \right)^m\norm{\mA_0 - \bar \mA_0}.
    \]
\end{lemma}
Hence, we obtain a linear convergence rate of $\left( 1-\sqrt{1-\rho} \right)$ instead of $\rho$. By virtue of that, we can set $m=\lceil \frac{1}{\sqrt{1-\rho}} \rceil$ to obtain a topology-independent iteration complexity.

\section{Convergence Analysis}

We present the complete proof in this section. In the sequel, $\|\cdot\|$ denotes the $\ell_2$-norm for vectors and the Frobenius norm for matrices. $\|\cdot\|_2$ denotes the spectral norm for matrices. $\mathbf{1}$ represents the all-one vector. We identify vectors at agent $i$ in the subscript and use the superscript for the algorithm step. For example, the optimization variable of agent $i$ at step $k$ is denoted as $x^k_i$, and $z^k_i$ is the corresponding dual variable. We use uppercase bold letters to represent the matrix that collects all the variables from agents (corresponding lowercase) as columns. To be specific,
\[
\mX_k = \left[x_{1}^{k}, \dots, x_{n}^{k}\right], \quad \mZ_k = \left[z_{1}^{k}, \dots, z_{n}^{k}\right], \quad  \mY_k = \left[y_{1}^{k},\dots, y_{n}^{k}\right], \quad \mV_{k+1} = \left[ v^{k+1}_1, \dots, v^{k+1}_n \right].
\]
We add an overbar to a letter to denote the average over all agents. For example, 
\[
\bar x^k = \avein x_{i}^{k} = \frac{1}{n}\mX_k \bfone, \quad \bar \mX_k = [\bar x^k, \dots, \bar x^k] = \bar x^k \bfonet = \frac{1}{n}\mX_k \bfone \bfonet
\]
Hence, the consensus errors for iterates $\{x^k_i\}$ and dual variables $\{z^k_i\}$ can be written as 
\[
\avein \norm{x^k_i - \bar x^k}^2  = \frac{1}{n}\norm{\mX_k - \bar \mX_k}^2,\quad \avein \norm{z^k_i - \bar z^k}^2  = \frac{1}{n}\norm{\mZ_k - \bar \mZ_k}^2.
\]
We denote $L_{\nabla F} = \underset{1\leq i\leq n}{\max} \{L_{\nabla F_i}\}$ for ease of presentation. Our proof heavily relies on the merit function below:
\begin{equation}\label{def: merit_fun}
    W(\bar x^k,\bar z^k) = \underbrace{\Phi(\bar x^{k}) - \Phi_*}_{\text{function value gap}} + \underbrace{\Psi(\bar x^k) - \eta(\bar x^{k}, \bar z^{k})}_{\text{primal convergence}} + \lambda \underbrace{\norm{\nabla F(\bar x^k) - \bar z^k}^2}_{\text{dual convergence}},
\end{equation}
where
\begin{equation}\label{def: eta}
    \eta(x, z) = \min_{y\in \realset^d}\left\{\<z,y-x> + \frac{1}{2\gamma}\|y-x\|^2 + \Psi(y)\right\}.
\end{equation}

\subsection{Technical Lemmas}

\begin{lemma}\label{lem: F_norm_ineq}
    For any $p, q, r\in \naturalset_+$ and matrix $\mA\in \realset^{p\times q}, \mB\in \realset^{q\times r}$, we have:
    \[
        \|\mA\mB\| \leq \min\left(\|\mA\|_2\cdot \|\mB\|, \|\mA\|\cdot \|\mB\T\|_2\right).
    \]
\end{lemma}

\begin{lemma}\label{lem: W_m}
    Suppose $\mW$ satisfies Assumption \ref{aspt:gossipMatrix}. For any $m\in \naturalset_+$, we have
    \[
        \norm{\mW^m - \frac{\bfone_n\bfonet_n}{n}}_2\leq \rho^m
    \]
\end{lemma}

\begin{lemma}\label{lem: cons_decay}
    Suppose we are given three sequences $\{a_n\}_{n=0}^{\infty}, \{b_n\}_{n=0}^{\infty},\ \{\tau_n\}_{n=-1}^{\infty},$ and a constant $r$ satisfying
    \begin{equation}\label{ineq: conditions_cons_decay}
        a_{k+1}\leq r a_k + b_{k},\ a_k\geq 0,\ b_k\geq 0,\  0 = \tau_{-1}\leq \tau_{k+1}\leq \tau_k\leq 1,
    \end{equation}
    for all $k\geq 0$. Then for any $K > 0$, we have
    \[
        \sum_{k=0}^{K}\tau_ka_k\leq \frac{1}{1-r}\left(\tau_0a_0 + \sum_{k=0}^{K}\tau_kb_k\right)
    \]
\end{lemma}
\begin{proof}
    Note that we have
    \begin{align*}
        (1-r)\sum_{k=0}^{K}\tau_ka_k \leq \sum_{k=0}^{K}\tau_k(a_k - a_{k+1} + b_k) = \sum_{k=0}^{K}(\tau_k - \tau_{k-1})a_k - \tau_Ka_{K+1} + \sum_{k=0}^{K}\tau_kb_k\leq \tau_0a_0 + \sum_{k=0}^{K}\tau_kb_k,
    \end{align*}
    where the inequalities use \eqref{ineq: conditions_cons_decay}, and the equality uses summation by parts.
\end{proof}

\begin{lemma}\label{lem:eta-smooth}
Let $\Psi: \realset^d \rightarrow \realset\cup\{+\infty\}$ be a closed proper convex function.
\begin{itemize}
\item[(a)] Let $\eta(x,z)$ be the function defined in \eqref{def: eta}. Then, $\nabla \eta$ is $C_\gamma$-Lipschitz continuous where
\begin{equation}
    C_{\gamma} = 2 \sqrt{(1+\frac{1}{\gamma})^2 + (1 + \frac{\gamma}{2})^2}.
\end{equation}
\item[(b)] For $x, z \in \realset^d$ and $\gamma\in\realset$, let $y_+ = \prox_{\Psi}^{\gamma}(x-\gamma z) = \underset{y\in \realset^d} {\argmin} \left\{\<z,y-x> + \frac{1}{2\gamma}\|y-x\|^2 + \Psi(y)\right\}$, then for any $y\in \realset^d$, we have
\begin{equation*}
    \Psi(y_+) - \Psi(y) \leq \< z + \gamma^{-1}(y_+ - x), y-y_+ >
\end{equation*}
\end{itemize}
\end{lemma}
\begin{proof}
We prove (a) at first. Recall that the Moreau envelope of a convex and closed function $\Psi$ multiplied by a scalar $\gamma$ is defined by
\begin{equation*}
    \text{env}_{\gamma \Psi}(x) = \underset{y\in\realset^d}{\min}\left\{\frac{1}{2\gamma}\norm{y-x}^2 + \Psi(y)\right\},
\end{equation*}
and its gradient is given by $\nabla \text{env}_{\gamma\Psi}(x) = \frac{1}{\gamma} (x - \prox_{\Psi}^{\gamma}(x))$ where $\prox_{\Psi}^{\gamma}(x)=\underset{y\in\realset^d}{\argmin}\left\{\frac{1}{2\gamma}\norm{y-x}^2 + \Psi(y)\right\}$. Note that $\eta(x, z) = \text{env}_{\gamma \Psi}\left(x-\gamma z\right) - \frac{\gamma}{2}\norm{z}^2$. Therefore, the partial gradients of $\eta$ are given by
\begin{equation}
    \nabla_x \eta (x, z) = -z - \gamma^{-1} \left(\prox_{\Psi}^{\gamma}\left(x-  \gamma z\right) - x\right), \quad \nabla_z \eta (x, z) =  \prox_{\Psi}^{\gamma}\left(x-  \gamma z\right) - x.
\end{equation}
Hence, for any $(x, z)$ and $(x', z')$,
\begin{align}
    \norm{\nabla \eta (x, z)  -\nabla \eta (x', z') } &\leq \norm{\nabla_x \eta (x, z) - \nabla_x \eta (x', z')} + \norm{\nabla_z \eta (x, z) - \nabla_z \eta (x', z')}\notag\\
    &\leq 2(1+1/\gamma) \norm{x - x'} + (2+\gamma) \norm{z- z'} \leq C_{\gamma} \norm{(x, z) - (x', z')}.\notag
\end{align}
To prove (b), denote the subdifferential of $\Psi(x)$ as $\partial \Psi(x)$. By the optimality condition, we have $\mathbf{0}$ is a subgradient of $H(y) = \<z,y-x> + \frac{1}{2\gamma}\|y-x\|^2 + \Psi(y)$ at $y_+$, i.e.,
\begin{equation*}
    \mathbf{0} \in z + \gamma^{-1}(y_+ - x) + \partial \Psi(y_+).
\end{equation*}
Hence, there exists a subgradient of $\Psi(y)$ at $y_+$, denoted by $\Tilde{\nabla}\Psi(y_+)$, such that
\begin{equation*}
    \Tilde{\nabla}\Psi(y_+) = - z - \gamma^{-1}(y_+ - x).
\end{equation*}
Finally, by the convexity of $\Psi$, we have for any $y\in\realset^d$,
\begin{equation*}
    \Psi(y) - \Psi(y_+) \geq \< \Tilde{\nabla} \Psi(y_+) , y-y_+> = \< - z - \gamma^{-1}(y_+ - x), y - y_+>,
\end{equation*}
which completes the proof.
\end{proof}

\subsection{Building Blocks of Main Proof}

The following lemma connects the consensus error of $\mY$ to the consensus errors of $\mX$ and $\mZ$.
\begin{lemma}\label{lem: consensus_y}
    Let $y^k_+ = \prox(\bar x^k - \gamma \bar z^k)$. Then for any $k\geq 0$ and $\gamma>0$, we have
    \begin{equation*}
        \norm{y^k_+ - \bar y^k}^2 + \frac{1}{n} \norm{\mY_k - \bar \mY_k}^2 = \frac{1}{n} \sum_{i=1}^{n} \norm{y_i^k - y^k_+}^2 \leq \frac{2}{n} \left\{\|\mX_{k} - \bar \mX_k\|^2 + \gamma^2\|\mZ_{k} - \bar \mZ_k\|^2 \right\}.
    \end{equation*}
\end{lemma}

\begin{proof}
    By the non-expansiveness of proximal operator, we have
    \begin{equation}\label{ineq: yi_y_plus_v2}
        \|y_{i}^{k} - y^k_+\| \leq \|x_{i}^{k} - \bar x^k - \gamma \left(z_{i}^{k} - \bar z^k\right)\|\leq \|x_{i}^{k} - \bar x^k\| + \gamma\|z_{i}^{k} - \bar z^k\|.
    \end{equation}
    Hence we know the consensus error of $y$ can be bounded
    \begin{align}
        &\frac{1}{n}\|\mY_k - \bar \mY^k\|^2 =\avein \|y_{i}^{k} - \bar y^k\|^2 = \avein \|y_{i}^{k} - y^k_+ + \avejn(y^k_+ - y_{j}^{k})\|^2 \notag\\
        =&\avein\|y_{i}^{k} - y^k_+\|^2 - \|\avejn\left(y_j^{k} - y_+^{k}\right)\|^2 \leq \avein\|y_{i}^{k} - y^k_+\|^2\notag\\
        \leq &\frac{2}{n} \left\{\|\mX_{k} - \bar \mX_k\|^2 + \gamma^2\|\mZ_{k} - \bar \mZ_k\|^2 \right\} \label{ineq: y_consensus}
    \end{align}
    where the third equality uses the fact that
    \[
        \avein\left\|v_i - \left(\avejn v_j\right)\right\|^2 = \avein\left\|v_i\right\|^2 - \left\|\avejn v_j\right\|^2
    \]
    for any vectors $v_i\ (1\leq i\leq n)$.
\end{proof}

The following technical lemma explicitly characterizes the consensus error.

\begin{lemma}[Conensus Error of Algorithm \ref{algo: Prox-DASA}: \texttt{Prox-DASA}]\label{lem: consensus}
Suppose Assumptions \ref{aspt:gossipMatrix}, \ref{aspt: Unbiasness}, \ref{aspt: independence}, \ref{aspt: Bounded Variance}, and \ref{aspt: Gradient heterogeneity} hold.  Let $\varrho(m) = \frac{(1+\rho^{2m}) \rho^{2m}}{(1-\rho^{2m})^2}$, and $\rho, m$ and $\alpha_k$ satisfy
    \begin{equation}\label{ineq: alpha_m_condition}
       \varrho(m)\alpha_k^2\leq \min\left\{\frac{1}{8}, \frac{1}{24L_{\nabla F}^2\gamma^2}\right\},\ 0 = \alpha_{-1}\leq\alpha_{k+1}\leq \alpha_k\leq 1
    \end{equation}
    for any $k\geq 0$.
    Then in Algorithm \ref{algo: Prox-DASA} for any $p\geq 0$, we have
    \begin{align*}
        \sum_{k=0}^{K}\frac{\alpha_k^p}{n}\E\left[\|\mX_k - \bar \mX_k\|^2\right]\leq 4\gamma^2(\sigma^2 + 3L_{\nabla F}^2\nu^2)\varrho(m)\sum_{k=0}^{K}\alpha_k^{p+2},\notag \\
        \sum_{k=0}^{K}\frac{\alpha_k^p}{n}\E\left[\|\mZ_k - \bar \mZ_k\|^2\right]\leq 4(\sigma^2 + 3L_{\nabla F}^2\nu^2)\varrho(m)\sum_{k=0}^{K}\alpha_k^{p+2},\notag.
    \end{align*}
\end{lemma}
\begin{proof}
By Assumption \ref{aspt:gossipMatrix}, the iterates in Algorithm \ref{algo: Prox-DASA} satisfy
\begin{equation}\label{eq: xd2}
    \begin{aligned}
       &\mX_{k+1} = (1-\alpha_k)\mX_k \mW^{m} + \alpha_k\mY_k \mW^{m},\ \bar x^{k+1} = (1-\alpha_k)\bar x^k + \alpha_k \bar y^k, \\ 
       &\mZ_{k+1} = (1-\alpha_k)\mZ_k \mW^{m} + \alpha_k \mV_{k+1} \mW^{m},\ \bar z^{k+1} = (1-\alpha_k)\bar z^k + \alpha_k \bar v^{k+1}. 
    \end{aligned}
\end{equation}
Hence, for the consensus error of iterates $\{x^k_i\} $, we have
    \begin{align}
        &\norm{\mX_{k+1} - \bar \mX_{k+1}}^2 \notag \\
        = &\left\|\bigg((1-\alpha_k)\left(\mX_k - \bar \mX_k\right) + \alpha_k\left(\mY_k - \bar \mY_k\right)\bigg)\left(\mW^m - \frac{\bfone\bfonet}{n}\right)\right\|^2 \notag\\
        \leq &\left\{\left(1 + \frac{1-\rho^{2m}}{2\rho^{2m}}\right)(1-\alpha_k)^2\norm{\mX_k - \bar \mX_k}^2 + \left(1 + \frac{2\rho^{2m}}{1-\rho^{2m}}\right)\alpha_k^2\norm{\mY_k - \bar \mY_k}^2\right\} \rho^{2m} \notag \\
        \leq &\frac{(1+\rho^{2m})}{2}\norm{\mX_k - \bar \mX_k}^2 + \frac{(1+\rho^{2m})\rho^{2m}}{1-\rho^{2m}}\alpha_k^2\norm{\mY_k - \bar \mY_k}^2, \label{ineq: x_consensus}
    \end{align}
    where the first inequality uses Lemma \ref{lem: F_norm_ineq} and \ref{lem: W_m}.  Combining \eqref{ineq: alpha_m_condition}, \eqref{ineq: x_consensus}, and Lemma \ref{lem: consensus_y}, we have
    \begin{align*}
        \E\left[\|\mX_{k+1} - \bar \mX_{k+1}\|^2\right]&\leq \frac{(1+\rho^{2m})}{2}\E\left[\|\mX_k - \bar \mX_k\|^2\right] + \frac{(1-\rho^{2m})}{4}\E\left[\|\mX_{k} - \bar \mX_k\|^2 + \gamma^2\|\mZ_{k} - \bar \mZ_k\|^2\right] \\
        &= \frac{(3 + \rho^{2m})}{4}\E\left[\|\mX_k - \bar \mX_k\|^2\right] + \frac{(1-\rho^{2m})\gamma^2}{4}\E\left[\|\mZ_{k} - \bar \mZ_k\|^2\right]
    \end{align*}
    Using Lemma \ref{lem: cons_decay} in the above inequality with $\tau_k = \frac{\alpha_k^p}{n}$ for any fixed $p\geq 0$ we know
    \begin{equation}\label{ineq: x_z_cons}
        \sum_{k=0}^{K}\frac{\alpha_k^p}{n}\E\left[\|\mX_k - \bar \mX_k\|^2\right]\leq \sum_{k=0}^{K}\frac{\gamma^2\alpha_k^p}{n}\E\left[\|\mZ_k - \bar \mZ_k\|^2\right].
    \end{equation}
    Similarly to \eqref{ineq: x_consensus}, we can obtain the following results on the consensus error of dual variables $\{z^k_i\}$:
    \begin{equation}\label{ineq: z_consensus}
        \norm{\mZ_{k+1} - \bar \mZ_{k+1}}^2 \leq \frac{(1+\rho^{2m})}{2}\norm{\mZ_k - \bar \mZ_k}^2 + \frac{(1+\rho^{2m})\rho^{2m}}{1-\rho^{2m}}\alpha_k^2\norm{\mV_{k+1} - \bar \mV_{k+1}}^2,
    \end{equation}
    Using \eqref{ineq: alpha_m_condition} and Lemma \ref{lem: cons_decay} in \eqref{ineq: z_consensus} with $\tau_k = \frac{\alpha_k^p}{n}$, we have
    \begin{align}
        \sum_{k=0}^{K}\frac{\alpha_k^p}{n}\E\left[\|\mZ_k - \bar \mZ_k\|^2\right]\leq 2\varrho(m)\sum_{k=0}^{K}\frac{\alpha_k^{p+2}}{n}\E\left[\|\mV_{k+1} - \bar \mV_{k+1}\|^2\right].\label{ineq: z_cons_v}
    \end{align}
    To bound $\|\mV_{k+1} - \bar \mV_{k+1}\|$ we first notice that
    \begin{align*}
        v_{i}^{k+1} - \bar v^{k+1}=&v_i^{k+1} - \E\left[v_i^{k+1}|\setF_k\right] - \avejn(v_{j}^{k+1} - \E\left[v_j^{k+1}|\setF_k\right]) \\
        +&\E\left[v_i^{k+1}|\setF_k\right] - \nabla F_i(\bar x^k) + \nabla F_i(\bar x^k) - \nabla F(\bar x^k) + \nabla F(\bar x^k) - \avejn\E\left[v_j^{k+1}|\setF_k\right] \\
        =&\left(1-\frac{1}{n}\right)(v_i^{k+1} - \E\left[v_i^{k+1}|\setF_k\right]) - \frac{1}{n}\sum_{j\neq i}(v_{j}^{k+1} - \E\left[v_j^{k+1}|\setF_k\right])\\
        + &\left(1-\frac{1}{n}\right)\left(\nabla F_i(x_i^{k}) - \nabla F_i(\bar x^k)\right) + \nabla F_i(\bar x^k) - \nabla F(\bar x^k) + \frac{1}{n}\sum_{j\neq i}\left(\nabla F_j(\bar x^k) - \nabla F_i(x_j^{k})\right)
    \end{align*}
    which gives
    \begin{align*}
        &\E\left[\|v_{i}^{k+1} - \bar v^{k+1}\|^2\right]\\
        = &\left(1-\frac{1}{n}\right)^2\E\left[\|v_i^{k+1} - \E\left[v_i^{k+1}|\setF_k\right]\|^2\right] + \frac{1}{n^2}\sum_{j\neq i}^{n}\E\left[\|v_{j}^{k+1} - \E\left[v_j^{k+1}|\setF_k\right]\|^2\right] \\
        + & \left\|\left(1-\frac{1}{n}\right)\left(\nabla F_i(x_i^{k}) - \nabla F_i(\bar x^k)\right) + \nabla F_i(\bar x^k) - \nabla F(\bar x^k) + \frac{1}{n}\sum_{j\neq i}\left(\nabla F_j(\bar x^k) - \nabla F_i(x_j^{k})\right)\right\|^2 \\
        \leq &\sigma^2 + 3L_{\nabla F}^2\left(\left(1-\frac{1}{n}\right)^2\|x_i^{k} - \bar x^{k}\|^2 + \nu^2 + \frac{1}{n}\sum_{j\neq i}\|x_j^k - \bar x^k\|^2\right),
    \end{align*}
    where the first equality uses Assumption \ref{aspt: independence}, and the second inequality uses Cauchy-Schwarz inequality, Assumptions \ref{aspt:lipschitz-gradient}, \ref{aspt: Bounded Variance}, and \ref{aspt: Gradient heterogeneity}. Hence we have
    \begin{equation}\label{ineq: V_cons}
        \E\left[\|\mV_{k+1} - \bar \mV_{k+1}\|^2\right] \leq 6L_{\nabla F}^2\E\left[\|\mX_k - \bar \mX_k\|^2\right] + n\sigma^2 + 3nL_{\nabla F}^2\nu^2.
    \end{equation}
    Combining \eqref{ineq: z_cons_v} and \eqref{ineq: V_cons}, we have
    \begin{equation}\label{ineq: z_x_cons}
        \begin{aligned}
            \sum_{k=0}^{K}\frac{\alpha_k^p}{n}\E\left[\|\mZ_k - \bar \mZ_k\|^2\right] \leq &2\varrho(m)\sum_{k=0}^{K} \left\{\frac{6L_{\nabla F}^2 \alpha_k^{p+2}}{n}\E\left[\|\mX_k - \bar \mX_k\|^2\right] + (\sigma^2 + 3L_{\nabla F}^2\nu^2)\sum_{k=0}^{K}\alpha_k^{p+2}\right\} \\
            \leq &\sum_{k=0}^{K}\left\{12 \varrho(m)\alpha_k^2 L_{\nabla F}^2 \gamma^2\right\}\frac{\alpha_k^{p}}{n\gamma^2}\E\left[\|\mX_k - \bar \mX_k\|^2\right] + 2(\sigma^2 + 3L_{\nabla F}^2\nu^2)\varrho(m)\sum_{k=0}^{K}\alpha_k^{p+2} \\
            \leq &\sum_{k=0}^{K}\frac{\alpha_k^p}{2n}\E\left[\|\mZ_k - \bar \mZ_k\|^2\right] + 2(\sigma^2 + 3L_{\nabla F}^2\nu^2)\varrho(m)\sum_{k=0}^{K}\alpha_k^{p+2},
        \end{aligned}
    \end{equation}
    where the second inequality uses \eqref{ineq: alpha_m_condition}. By \eqref{ineq: x_z_cons} and \eqref{ineq: z_x_cons} we can finally obtain that
    \begin{align}
        \sum_{k=0}^{K}\frac{\alpha_k^p}{n}\E\left[\|\mX_k - \bar \mX_k\|^2\right]\leq 4\gamma^2(\sigma^2 + 3L_{\nabla F}^2\nu^2)\varrho(m)\sum_{k=0}^{K}\alpha_k^{p+2},\label{ineq: x_cons_final}, \\
        \sum_{k=0}^{K}\frac{\alpha_k^p}{n}\E\left[\|\mZ_k - \bar \mZ_k\|^2\right]\leq 4(\sigma^2 + 3L_{\nabla F}^2\nu^2)\varrho(m)\sum_{k=0}^{K}\alpha_k^{p+2},\label{ineq: z_cons_final}.
    \end{align}
\end{proof}

\begin{lemma}[Conensus Error of Algorithm \ref{algo: Prox-DASA-GT}: \texttt{Prox-DASA-GT}]\label{lem: consensus_gt}
    Suppose Assumptions \ref{aspt:gossipMatrix}, \ref{aspt: Unbiasness}, \ref{aspt: Bounded Variance} and \ref{aspt: independence} hold. Let $\varrho(m) = \frac{(1+\rho^{2m}) \rho^{2m}}{(1-\rho^{2m})^2}$, and $\rho, m$ and $\alpha_k$ satisfy
    \begin{equation}\label{ineq: alpha_m_condition_gt}
        \varrho(m)\alpha_k^2\leq \frac{1}{8},\ \varrho(m)\alpha_k\leq \frac{1}{9L_{\nabla F}\gamma},\ 0 = \alpha_{-1}\leq\alpha_{k+1}\leq \alpha_k\leq 1
    \end{equation}
    for any $k\geq 0$, and the initialization satisfies $u_i^0 = v_i^0= 0$ for all $i$. Then in Algorithm \ref{algo: Prox-DASA-GT} for any $p\geq 0$ we have
    \begin{align*}
        &\sum_{k=0}^{K}\frac{\alpha_k^p}{n}\E\left[\|\mX_k - \bar \mX_k\|^2\right]\leq 40 \gamma^2 \varrho(m)^2 \sum_{k=0}^{K}\alpha_k^{p+2}\left\{L_{\nabla F}^2 \alpha_k^2\E\left[\|\bar x^k - \bar y^k\|^2\right] +2\sigma^2\right\},\\
        &\sum_{k=0}^{K}\frac{\alpha_k^p}{n}\E\left[\|\mZ_k - \bar \mZ_k\|^2\right]\leq 40 \varrho(m)^2 \sum_{k=0}^{K}\alpha_k^{p+2}\left\{L_{\nabla F}^2 \alpha_k^2\E\left[\|\bar x^k - \bar y^k\|^2\right] +2\sigma^2\right\}.
    \end{align*}
\end{lemma}
\begin{proof}
    The updates in Algorithm \ref{algo: Prox-DASA-GT} take the form:
    \begin{equation}\label{eq: xd2_gt}
        \begin{aligned}
           &\mX_{k+1} = (1-\alpha_k)\mX_k \mW^{m} + \alpha_k\mY_k \mW^{m},\ \bar x^{k+1} = (1-\alpha_k)\bar x^k + \alpha_k \bar y^k, \\ 
           &\mU_{k+1} = \mU_k \mW^{m} + (\mV_{k+1} - \mV_k) \mW^{m},\ \bar u^{k+1} = \bar u^k + \bar v^{k+1} - \bar v^k, \\
           &\mZ_{k+1} = (1-\alpha_k)\mZ_k \mW^{m} + \alpha_k \mU_{k} \mW^{m},\ \bar z^{k+1} = (1-\alpha_k)\bar z^k + \alpha_k \bar u^{k}. 
        \end{aligned}
    \end{equation}
    Setting $u_i^0 = v_i^0$, we can prove by induction that $\bar u^k = \bar v^k$. To analyze the consensus error of $\mU_k$, we first notice:
    \begin{align*}
        &\mU_{k+1} - \bar \mU_{k+1} \\
        = & \left(\mU_k - \bar \mU_k + \mV_{k+1} - \mV_k - \bar \mV^{k+1} + \bar \mV^k\right)\left(\mW^m - \frac{\bfone\bfonet}{n}\right)\\
        = &\left(\mU_k - \bar \mU_k + \left(\mV_{k+1} - \mV_k\right)\left(\mI - \frac{\bfone\bfonet}{n}\right)\right)\left(\mW^m - \frac{\bfone\bfonet}{n}\right)
    \end{align*}
    which gives
    \begin{align*}
        &\|\mU_{k+1} - \bar \mU_{k+1}\|^2\\
        \leq &\left\{\left(1 + \frac{1-\rho^{2m}}{2\rho^{2m}}\right)\norm{\mU_k - \bar \mU_k}^2 + \left(1 + \frac{2\rho^{2m}}{1-\rho^{2m}}\right)\norm{\mV_{k+1} -  \mV_k}^2\right\} \rho^{2m} \\
        = &\frac{(1+\rho^{2m})}{2}\norm{\mU_k - \bar \mU_k}^2 + \frac{(1+\rho^{2m})\rho^{2m}}{1-\rho^{2m}}\norm{\mV_{k+1} -\mV_k}^2.
    \end{align*}
    Using Lemma \ref{lem: cons_decay}, we know for any $k\geq 0$ and $p\geq 0$,
    \begin{equation}\label{ineq: U_V_cons}
        \sum_{k=0}^{K}\alpha_k^p\|\mU_{k} - \bar \mU_{k}\|^2\leq2\varrho(m) \sum_{k=0}^{K}\alpha_k^p\norm{\mV_{k+1} -\mV_k}^2.
    \end{equation}
    Note that we also have
    \begin{align*}
        \mV_{k+1} - \mV_k = &\mV_{k+1} - \E\left[\mV_{k+1}|\setF_k\right] - \left(\mV_k - \E\left[\mV_k|\setF_{k-1}\right]\right) \\
        &+ \E\left[\mV_{k+1}|\setF_k\right] - \nabla \mF(\bar x^k) + \nabla \mF(\bar x^k) - \nabla \mF(\bar x^{k-1}) + \nabla \mF(\bar x^{k-1}) - \E\left[\mV_k|\setF_{k-1}\right]
    \end{align*}
    where we overload the notation and define $\nabla \mF(x) = [\nabla F_1(x), ..., \nabla F_n(x)]$. Hence we know
    \begin{equation}\label{ineq: delta_V}
        \begin{aligned}
            &\E\left[\|\mV_{k+1} - \mV_k\|^2\right]\\
        \leq & 5\bigg\{\E\left[\|\mV_{k+1} - \E\left[\mV_{k+1}|\setF_k\right]\|^2\right] + \E\left[\|\mV_k - \E\left[\mV_k|\setF_{k-1}\right]\|^2\right] + \E\left[\sum_{i=1}^{n}\|\nabla F_i(x_i^k) - \nabla F_i(\bar x^k)\|^2\right] \\
        + & \E\left[\sum_{i=1}^{n}\|\nabla F_i(\bar x^k) - \nabla F_i(\bar x^{k-1})\|^2\right] + \E\left[\sum_{i=1}^{n}\|\nabla F_i(x_i^{k-1}) - \nabla F_i(\bar x^{k-1})\|^2\right]\bigg\} \\
        \leq &5\left(2n\sigma^2 + L_{\nabla F}^2\E\left[\|\mX_k - \bar \mX_k\|^2 + \|\mX^{k-1} - \bar \mX^{k-1}\|^2 + n\alpha_{k-1}^2\|\bar x^{k-1} - \bar y^{k-1}\|^2\right]\right)
        \end{aligned}
    \end{equation}
    where the first inequality uses Cauchy-Schwarz inequality, and the second inequality uses Lipschitz continuity of $\nabla f_i$ and \eqref{eq: xd2_gt}. For simplicity we set $x_i^{-1} = y_i^{-1} = 0$ for all $i$ so that it is easy to check the above inequality holds for all $k\geq 0$. Using \eqref{ineq: U_V_cons} and \eqref{ineq: delta_V} we know:
    \begin{align}
        &\sum_{k=0}^{K}\frac{\alpha_k^p}{n}\|\mU_{k} - \bar \mU_{k}\|^2\\
        \leq &\frac{10\varrho(m)}{n} \sum_{k=0}^{K}\alpha_k^p\left(2n\sigma^2 + L_{\nabla F}^2\E\left[\|\mX_k - \bar \mX_k\|^2 + \|\mX^{k-1} - \bar \mX^{k-1}\|^2 + n\alpha_{k-1}^2\|\bar x^{k-1} - \bar y^{k-1}\|^2\right]\right). \notag \\
            \leq &\frac{20L_{\nabla F}^2\varrho(m)}{n}\sum_{k=0}^{K}\alpha_k^p\E\left[\|\mX_k - \bar \mX_k\|^2\right] + 10L_{\nabla F}^2\varrho(m)\sum_{k=0}^{K}\alpha_k^{p+2}\E\left[\|\bar x^k - \bar y^k\|^2\right] +20\sigma^2 \varrho(m)\sum_{k=0}^{K}\alpha_k^{p} \label{ineq: U_cons_gt},
    \end{align}
    where the third inequality uses \eqref{ineq: alpha_m_condition_gt}.
    For other consensus error terms we follow the same proof in Lemma \ref{lem: consensus} to get
    \begin{align}
        &\norm{\mX_{k+1} - \bar \mX_{k+1}}^2\leq \frac{(1+\rho^{2m})}{2}\norm{\mX_k - \bar \mX_k}^2 + \frac{(1+\rho^{2m})\rho^{2m}}{1-\rho^{2m}}\alpha_k^2\norm{\mY_k - \bar \mY_k}^2, \label{ineq: X_cons_gt}\\
        &\norm{\mY_k - \bar \mY_k}^2 \leq 2(\norm{\mX_k - \bar \mX_k}^2 + \gamma^2\norm{\mZ_k - \bar \mZ_k}^2), \label{ineq: Y_cons_gt}\\
        &\norm{\mZ_{k+1} - \bar \mZ_{k+1}}^2 \leq \frac{(1+\rho^{2m})}{2}\norm{\mZ_k - \bar \mZ_k}^2 + \frac{(1+\rho^{2m})\rho^{2m}}{1-\rho^{2m}}\alpha_k^2\norm{\mU_k - \bar \mU_k}^2 \label{ineq: Z_cons_gt}.
    \end{align}
    Hence we know \eqref{ineq: x_z_cons} still holds:
    \begin{equation}\label{ineq: x_z_cons_gt}
        \sum_{k=0}^{K}\frac{\alpha_k^p}{n}\E\left[\|\mX_k - \bar \mX_k\|^2\right]\leq \sum_{k=0}^{K}\frac{\gamma^2\alpha_k^p}{n}\E\left[\|\mZ_k - \bar \mZ_k\|^2\right].
    \end{equation}
    Applying Lemma \eqref{lem: cons_decay} in \eqref{ineq: Z_cons_gt} with $\tau_k = \frac{\alpha_k^p}{n}$, we have
    \begin{equation}\label{ineq: z_u_cons_gt}
        \sum_{k=0}^{K}\frac{\alpha_k^p}{n}\E\left[\|\mZ_k - \bar \mZ_k\|^2\right]\leq 2\varrho(m)\sum_{k=0}^{K}\frac{\alpha_k^{p+2}}{n}\E\left[\|\mU_k - \bar \mU_k\|^2\right].
    \end{equation}
    The above two inequalities together with \eqref{ineq: U_cons_gt} and \eqref{ineq: alpha_m_condition_gt} imply
    \begin{align*}
        &\sum_{k=0}^{K}\frac{\alpha_k^p}{n}\E\left[\|\mX_k - \bar \mX_k\|^2\right]\leq  2\varrho(m)\gamma^2\sum_{k=0}^{K}\frac{\alpha_k^{p+2}}{n}\E\left[\|\mU_k - \bar \mU_k\|^2\right] \\
        \leq& \sum_{k=0}^{K}\left\{40L_{\nabla F}^2\gamma^2\varrho(m)^2\alpha_k^2\right\}\frac{\alpha_k^{p}}{n}\E\left[\|\mX_k - \bar \mX_k\|^2\right] + 20 \gamma^2 \varrho(m)^2 \sum_{k=0}^{K}\alpha_k^{p+2}\left\{L_{\nabla F}^2 \alpha_k^2\E\left[\|\bar x^k - \bar y^k\|^2\right] +2\sigma^2\right\}\\
        \leq&\frac{1}{2}\sum_{k=0}^{K}\frac{\alpha_k^{p}}{n}\E\left[\|\mX_k - \bar \mX_k\|^2\right] + 20 \gamma^2 \varrho(m)^2 \sum_{k=0}^{K}\alpha_k^{p+2}\left\{L_{\nabla F}^2 \alpha_k^2\E\left[\|\bar x^k - \bar y^k\|^2\right] +2\sigma^2\right\},
    \end{align*}
    which gives
    \begin{equation}\label{ineq: X_cons_final_gt}
        \sum_{k=0}^{K}\frac{\alpha_k^p}{n}\E\left[\|\mX_k - \bar \mX_k\|^2\right]\leq 40 \gamma^2 \varrho(m)^2 \sum_{k=0}^{K}\alpha_k^{p+2}\left\{L_{\nabla F}^2 \alpha_k^2\E\left[\|\bar x^k - \bar y^k\|^2\right] +2\sigma^2\right\}.
    \end{equation}
    Combining \eqref{ineq: alpha_m_condition_gt}, \eqref{ineq: U_cons_gt}, \eqref{ineq: z_u_cons_gt}, and \eqref{ineq: X_cons_final_gt}, we obtain that
    \begin{align*}
        &\sum_{k=0}^{K}\frac{\alpha_k^p}{n}\E\left[\|\mZ_k - \bar \mZ_k\|^2\right] \leq 2\varrho(m)\sum_{k=0}^{K}\frac{\alpha_k^{p+2}}{n}\E\left[\|\mU_k - \bar \mU_k\|^2\right] \notag\\
        \leq&\frac{1}{2\gamma^2}\sum_{k=0}^{K}\frac{\alpha_k^{p}}{n}\E\left[\|\mX_k - \bar \mX_k\|^2\right] + 20 \varrho(m)^2 \sum_{k=0}^{K}\alpha_k^{p+2}\left\{L_{\nabla F}^2 \alpha_k^2\E\left[\|\bar x^k - \bar y^k\|^2\right] +2\sigma^2\right\},\\
        \leq &40 \varrho(m)^2 \sum_{k=0}^{K}\alpha_k^{p+2}\left\{L_{\nabla F}^2 \alpha_k^2\E\left[\|\bar x^k - \bar y^k\|^2\right] +2\sigma^2\right\}.
    \end{align*}

\end{proof}

\begin{lemma}[Basic Inequalities of Dual Convergence]\label{lem: z_nabla_F}
\begin{equation}\label{def: delta_r}
    \begin{split}
        \delta^{k} &= \frac{\nabla F(\bar x^k) - \nabla F(\bar x^{k+1})}{\alpha_k} +  
        \frac{1}{n}\sum_{i=1}^{n} \nabla F_i(x^k_i) - \nabla F(\bar x^k),\quad \bar\Delta^{k+1} = \bar v^{k+1} - \frac{1}{n}\sum_{i=1}^{n} \nabla F_i(x^k_i).
    \end{split}
\end{equation}
    Under Assumption \ref{aspt:lipschitz-gradient}, we have
    \begin{equation}\label{ineq:zbar-mse-recursive}
    \begin{split}
        \|\bar z^{k+1} - & \nabla F(\bar x^{k+1})\|^2 \leq (1-\alpha_k) \norm{\bar z^k - \nabla F(\bar x^k)}^2 +  2L_{\nabla F}^2 \alpha_k\norm{\bar x^k -  \bar y^k}^2  + \alpha_k^2 \norm{\bar\Delta^{k+1}}^2\\
        &+ \frac{2L_{\nabla F}^2\alpha_k}{n}\norm{\mX_k - \bar \mX_k}^2  + 2\< \alpha_k \bar\Delta^{k+1}, (1-\alpha_k) \left(\bar z^k - \nabla F(\bar x^k)\right) + \alpha_k\delta^{k}>,
    \end{split}
    \end{equation}
    and
    \begin{equation}\label{ineq:zbar-diff}
    \begin{split}
        \norm{\bar z^{k+1} - \bar z^{k}}^2 & \leq  \alpha_k^2 \bigg\{2\norm{\nabla F(\bar x^k) - \bar z^k}^2 +  \frac{2 L_{\nabla F}^2}{n}\norm{\mX_k - \bar \mX_k}^2 + \norm{\bar \Delta^{k+1}}^2 \\
        &\hspace{16em} + 2\<\bar \Delta^{k+1}, \frac{1}{n}\sum_{i=1}^{n} \nabla F_i(x^k_i) - \bar z^k>\bigg\}.
    \end{split}
    \end{equation}
\end{lemma}
\begin{proof}
    By definitions in \eqref{def: delta_r}, we have
    \begin{equation*}
        \bar z^{k+1} - \nabla F(\bar x^{k+1}) = (1-\alpha_k) \left(\bar z^k - \nabla F(\bar x^k)\right) + \alpha_k\delta^{k}  + \alpha_k \bar\Delta^{k+1},
    \end{equation*}
    Hence, we can get
    \begin{align*}
        &\norm{\bar z^{k+1} - \nabla F(\bar x^{k+1})}^2 \\
        &= \norm{(1-\alpha_k) \left(\bar z^k - \nabla F(\bar x^k)\right) + \alpha_k\delta^{k}}^2  + \alpha_k^2 \norm{\bar\Delta^{k+1}}^2 + 2\< \alpha_k \bar\Delta^{k+1}, (1-\alpha_k) \left(\bar z^k - \nabla F(\bar x^k)\right) + \alpha_k\delta^{k}>\\
        &\leq (1-\alpha_k) \norm{\bar z^k - \nabla F(\bar x^k)}^2 + \alpha_k \norm{\delta^{k}}^2 + \alpha_k^2 \norm{\bar\Delta^{k+1}}^2 + 2\< \alpha_k \bar\Delta^{k+1}, (1-\alpha_k) \left(\bar z^k - \nabla F(\bar x^k)\right) + \alpha_k\delta^{k}>
    \end{align*}
    where the inequality uses the convexity of $\|\cdot\|^2$. In addition, we have
    \begin{align*}
        \norm{\delta^{k}}^2 &\leq 2\norm{\frac{\nabla F(\bar x^k) - \nabla F(\bar x^{k+1})}{\alpha_k}}^2 + 2\norm{\avein\left(\nabla F_i(x_i^k) - \nabla F_i(\bar x^k)\right)}^2 \\
        &\leq 2L_{\nabla F}^2 \norm{\bar x^k - \bar y^k}^2 + \frac{2L_{\nabla F}^2}{n}\norm{\mX_k - \bar \mX_k}^2, \\
    \end{align*}
    which completes the proof of \eqref{ineq:zbar-mse-recursive}. The inequality \eqref{ineq:zbar-diff} can be proved similarly by noting that
    \begin{align}
        & \norm{\bar z^{k+1} - \bar z^k}^2  = \alpha_k^2 \norm{-\bar z^k + \bar v^{k+1}}^2\notag \\
        &= \alpha_k^2 \norm{(\nabla F(\bar x^k) - \bar z^k) + \left(\frac{1}{n}\sum_{i=1}^{n}\left(\nabla F_i(x^k_i) - \nabla F_i(\bar x^k)\right)\right) + \alpha_k \bar \Delta^{k+1}}^2 \notag\\
        &= \alpha_k^2 \bigg\{\norm{(\nabla F(\bar x^k) - \bar z^k) + \left(\frac{1}{n}\sum_{i=1}^{n}\left(\nabla F_i(x^k_i) - \nabla F_i(\bar x^k)\right)\right)}^2  + \norm{\bar \Delta^{k+1}}^2 + 2\<\bar \Delta^{k+1}, \frac{1}{n}\sum_{i=1}^{n} \nabla F_i(x^k_i) - \bar z^k>\bigg\}.\notag
    \end{align}
\end{proof}

\begin{lemma}[]\label{lem: psi_yplus_ybar} Under Assumption \ref{aspt:Psi},
\begin{equation}\label{ineq: delta_psi_2}  
    \Psi(\bar y^k)  - \Psi(y_+^k)\leq \< \bar z^k + \gamma^{-1}(\bar y^k - \bar x^k), y^k_+ - \bar y^k  > + \frac{\gamma}{2n}\norm{\mZ_k - \bar \mZ_k}^2 + \frac{\gamma^{-1}}{2n}\norm{\mX_k - \bar \mX_k }^2. 
\end{equation}
\end{lemma}
\begin{proof}
By the convexity of $\Psi$ and part (b) of Lemma \ref{lem:eta-smooth}, we have
\begin{align}
    \Psi(\bar y^k) & - \Psi(y_+^k) \overset{\text{cvx}}{\leq} \frac{1}{n}\sum_{i=1}^{n} \left(\Psi(y_i^k) - \Psi(y^k_+)\right) \overset{\text{Lemma \ref{lem:eta-smooth} (b)}}{\leq} \frac{1}{n}\sum_{i=1}^{n}\< z_i^k + \gamma^{-1} (y_i^k - x_i^k) , y^k_+ - y_i^k >\notag\\
    &= \< \bar z^k + \gamma^{-1}(\bar y^k - \bar x^k), y^k_+ - \bar y^k  > + \frac{1}{n}\sum_{i=1}^{n} \< z_i^k - \bar z^k + \gamma^{-1} (y_i^k -\bar y^k + \bar x^k - x_i^k)  ,  \bar y^k - y_i^k>\notag\\
    &\leq \< \bar z^k + \gamma^{-1}(\bar y^k - \bar x^k), y^k_+ - \bar y^k  > + \frac{\gamma}{2n}\norm{\mZ_k - \bar \mZ_k}^2 + \frac{1}{2n\gamma}\norm{\mX_k - \bar \mX_k}^2. \notag
\end{align}
The equality above comes from the fact that for sequences $\{a_i\}_{1\leq i\leq n}, \{b_i\}_{1\leq i\leq n} \in \realset^d$, we have
$$\sum_{i=1}^{n} \<a_i - \frac{1}{n}\sum_{i=1}^{n}a_i,b_i - \frac{1}{n} \sum_{i=1}^{n}b_i> = \sum_{i=1}^{n} \<a_i, b_i> - \left(\frac{1}{n}\sum_{i=1}^{n}a_i\right)\left( \frac{1}{n}\sum_{i=1}^{n}b_i\right).$$ 
The last inequality above is obtained by Young's inequalities:
\begin{align}
    \< z^k_i - \bar z^k, \bar y^k - y^k_i> &\leq  \frac{\gamma}{2} \norm{z^k_i - \bar z^k}^2 + \frac{1}{2\gamma} \norm{y^k_i - \bar y^k}^2, \notag\\
    \gamma^{-1}\< \bar x^k - x^k_i, \bar y^k - y^k_i> &\leq \frac{1}{2\gamma}\norm{x^k_i - \bar x^k}^2  + \frac{1}{2\gamma} \norm{y^k_i - \bar y^k}^2.\notag
\end{align}
\end{proof}

\begin{lemma}[Basic Lemma of Merit Function Difference]\label{lem: main-merit-func-diff}
Let $W(\bar x^k, \bar z^k)$ be the merit function defined in \eqref{def: merit_fun} with $\lambda = \frac{\gamma^{-1}}{8 L_{\nabla F}^2}$. Under Assumption \ref{aspt:lipschitz-gradient}, \ref{aspt:Psi}, for any $k \geq 0$, setting  $\alpha_k \leq \min\{\frac{\gamma^{-1}}{8 L_{\nabla F}}, \frac{\gamma^{-1}}{8 C_\gamma}, \frac{\gamma^{-1}}{32C_\gamma L_{\nabla F}^2}\}$,  we have
\begin{equation*}
     W(\bar x^{k+1}, \bar z^{k+1}) - W(\bar x^{k}, \bar z^{k}) \leq - \alpha_k \left\{\Theta^k  +  \Upsilon^k + \alpha_k \Lambda^k +  r^{k+1} \right\},
\end{equation*}
where
\begin{equation}\label{def:Theta-Lambda-Upsilon-r}
\begin{split}
    \Theta^k = & \left\{\frac{\gamma^{-1}}{4} \|\bar x^k  - \bar y^k\|^2 + \frac{\lambda}{4} \norm{\nabla F(\bar x^k) - \bar z^k}^2 \right\},  \quad \Lambda^k = \left\{\frac{C_\gamma + 2\lambda}{2}\norm{\bar \Delta^{k+1}}^2\right\},\\
    \Upsilon^k = & \left\{\frac{2\gamma(1 + 4\gamma^2L_{\nabla F}^2)}{n}\norm{\mZ_k - \bar \mZ_k}^2 + \frac{2\left(\gamma^{-1} + 3\gamma L_{\nabla F}^2 \right)}{n}\norm{\mX_k - \bar \mX_k}^2\right\},\\
    r^{k+1}  = & \< \bar \Delta^{k+1},  \bar x^k - y^k_+ + C_\gamma \alpha_k\left(\frac{1}{n}\sum_{i=1}^{n} \nabla F_i(x^k_i) - \bar z^k\right) + 2\lambda \left( (1-\alpha_k) \left(\bar z^k - \nabla F(\bar x^k)\right) + \alpha_k\delta^{k}\right)>.
\end{split}
\end{equation}
\end{lemma}
\begin{proof}
    By the smoothness of $F$ and $\eta$, we have
\begin{align}
    &F(\bar x^{k+1}) - F(\bar x^k) \notag\\
    \leq &\<\nabla F(\bar x^k), \bar x^{k+1} - \bar x^k> + \frac{L_{\nabla F}}{2}\|\bar x^{k+1} - \bar x^k\|^2 = -\alpha_k\<\nabla F(\bar x^k), \bar x^k - \bar y^k> + \frac{L_{\nabla F}\alpha_k^2}{2}\|\bar x^k - \bar y^k\|^2\label{ineq: delta_f}\\
    &\eta(\bar x^k, \bar z^k) - \eta(\bar x^{k+1}, \bar z^{k+1})\notag\\
    \leq & \<-\bar z^k - \gamma^{-1} (y^k_+ - \bar x^k), \bar x^k - \bar x^{k+1}> + \<y^k_+ - \bar x^k, \bar z^k - \bar z^{k+1}> + \frac{C_\gamma}{2}\left(\|\bar x^{k+1}- \bar x^k\|^2 + \|\bar z^{k+1} - \bar z^k\|^2\right) \notag\\
    = &2\alpha_k\<\bar z^k, y^k_+-\bar x^k> + \gamma^{-1}\alpha_k\|\bar x^k - y^k_+\|^2 + \alpha_k\<\bar v^{k+1}, \bar x^k - \bar y^k>  \notag \\
    & \quad + \alpha_k \<\bar z^k + \gamma^{-1}(y^k_+-\bar x^k) + \bar v^{k+1} , \bar y^k - y^k_+>  + \frac{C_\gamma}{2}\left(\alpha_k^2\|\bar x^k-\bar y^k\|^2 + \|\bar z^{k+1} - \bar z^k\|^2\right).\label{ineq: delta_eta_1}
\end{align}
Since $y^{k}_{+}$ is the minimizer of a $1/\gamma$-strongly convex function, i.e.,
\begin{equation*}
    \<\bar z^k, y^k_+ - \bar x^k> + \frac{1}{2\gamma}\|y^k_+ - \bar x^k\|^2 +\Psi(y^k_+) \leq \Psi(\bar x^k)  - \frac{1}{2\gamma}\|y^k_+ - \bar x^k\|^2,
\end{equation*}
which together with \eqref{ineq: delta_eta_1} gives
\begin{align}
    &\eta(\bar x^k, \bar z^k) - \eta(\bar x^{k+1},\bar z^{k+1}) \notag \\
    \leq &-\gamma^{-1}\alpha_k\|\bar x^k - y^k_+\|^2 + \alpha_k\<\bar v^{k+1}, \bar x^k - \bar y^k> + \alpha_k\< \bar z_k + \gamma^{-1}(y^k_+-\bar x^k) + \bar v^{k+1}, \bar y^k - y^k_+> \notag\\
    & + 2\alpha_k \left(\Psi(\bar x^k) - \Psi(y^k_+)\right) +\frac{C_\gamma}{2}\left(\|\bar x^{k+1}-\bar x^k\|^2 + \|\bar z^{k+1} - \bar z^k\|^2\right). \label{ineq: delta_eta_2}
\end{align}
By the convexity of $\Psi$, we have
\begin{equation}\label{ineq: delta_psi_1}
    \Psi(\bar x^{k+1}) - \Psi(\bar x^k) \leq (1-\alpha_k) \Psi(\bar x^k) + \alpha_k \Psi(\bar y^k) - \Psi(\bar x^k) = \alpha_k \left( \Psi(\bar y^k) - \Psi(\bar x^k_i) \right).
\end{equation}
Combining \eqref{ineq: delta_f}, \eqref{ineq: delta_eta_2}, and \eqref{ineq: delta_psi_1}, we have
\begin{equation}\label{ineq: merit_diff_1}
    \begin{aligned}
        &\left[\Phi(\bar x^{k+1}) + \Psi(\bar x^{k+1}) - \eta(\bar x^{k+1}, \bar z^{k+1})\right] - \left[\Phi(\bar x^{k}) + \Psi(\bar x^k) - \eta(\bar x^{k}, \bar z^{k})\right]  \\
    \leq &- \gamma^{-1}\alpha_k\|\bar x^k - y^k_{+}\|^2  + \alpha_k \< \bar v^{k+1}-\nabla F(\bar x^k), \bar x^k - \bar y^k> + 2\alpha_k(\Psi(\bar y^k) - \Psi(y_+^k))\\
    & + \alpha_k\< \bar z^k + \gamma^{-1}(y^k_+ -\bar x^k) + 
    \bar v^{k+1}, \bar y^k - y^k_+> +\frac{(L_{\nabla F}+C_\gamma)\alpha_k^2}{2}\|\bar x^k - \bar y^k\|^2 + \frac{C_{\gamma}}{2} \|\bar z^{k+1} - \bar z^k\|^2.
    \end{aligned}
\end{equation}
Removing non-smooth terms in \eqref{ineq: merit_diff_1} using \eqref{ineq: delta_psi_2} in Lemma \ref{lem: psi_yplus_ybar}, and re-organizing \eqref{ineq: merit_diff_1} using the decomposition that $\bar z^{k+1} - \bar z^k = \alpha_k(-\bar z^k + \bar v^{k+1}) = \alpha_k (\nabla F(\bar x^k) - \bar z^k) + \alpha_k (\frac{1}{n}\sum_{i=1}^{n}(\nabla F_i(x^k_i) - \nabla F_i(\bar x^k))) + \alpha_k \bar \Delta^{k+1}$, we can get
\begin{align}
     &\left[\Phi(\bar x^{k+1}) + \Psi(\bar x^{k+1}) - \eta(\bar x^{k+1}, \bar z^{k+1})\right] - \left[\Phi(\bar x^{k}) + \Psi(\bar x^k) - \eta(\bar x^{k}, \bar z^{k})\right]  \notag\\
    \leq & \underbrace{\gamma^{-1}\alpha_k\bigg\{-\|\bar x^k  - y^k_{+}\|^2  + \< (y^k_+ -\bar y^k) + (\bar x^k - \bar y^k) , \bar y^k - y^k_+> \bigg\}}_{\varkappa_1}  \notag \\
    & + \underbrace{\alpha_k \<  \frac{1}{n}\sum_{i=1}^{n} \left( \nabla F_i (x^k_i) - \nabla F_i(\bar x^k) \right), \bar x^k - y^k_+ >}_{\varkappa_2} + \underbrace{\alpha_k\< \nabla F(\bar x^k) - \bar z^k,   \bar y^k - y^k_+>}_{\varkappa_3} +  \alpha_k \< \bar \Delta^{k+1},  \bar x^k - y^k_+>  \notag\\
    & \frac{(L_{\nabla F}+C_\gamma)\alpha_k^2}{2}\|\bar x^k - \bar y^k\|^2 + \underbrace{\frac{C_{\gamma}}{2} \|\bar z^{k+1} - \bar z^k\|^2}_{\varkappa_4} + \frac{\gamma\alpha_k}{n}\norm{\mZ_k - \bar \mZ_k}^2 + \frac{\gamma^{-1}\alpha_k}{n}\norm{\mX_k - \bar \mX_k}^2.\notag
\end{align}
To further simplify the above inequalities, we analyze the terms $\varkappa_1, \varkappa_2, \varkappa_3, \varkappa_4$ separately as follows:
\begin{align}
\varkappa_1 = & \gamma^{-1}\alpha_k \left\{- \norm{\bar x^k - \bar y^k}^2 - \< \bar x^k - \bar y^k, \bar y^k - y^k_+ > - 2 \norm{\bar y^k - y^k_{+}}^2\right\} \leq -\frac{7\gamma^{-1}\alpha_k}{8} \norm{\bar x^k - \bar y^k}^2,\notag\\
\varkappa_2 \leq &\, 2\gamma\alpha_k\norm{\frac{1}{n}\sum_{i=1}^{n} \left( \nabla F_i (x^k_i) - \nabla F_i(\bar x^k) \right)}^2 + \frac{\gamma^{-1}\alpha_k}{8} \norm{ \bar x^k - y^k_+}^2\notag\\
\leq &\, \frac{2\gamma \alpha_k L_{\nabla F}^2}{n} \norm{\mX_k - \bar \mX_k}^2+ \frac{\gamma^{-1}\alpha_k}{4} \norm{ \bar x^k - \bar y^k}^2 + \frac{\gamma^{-1}\alpha_k}{4} \norm{ \bar y^k - y^k_+}^2\notag,\\
\varkappa_3 \leq &\, \frac{\lambda \alpha_k}{2} \norm{\nabla F(\bar x^k) - \bar z^k}^2 + \frac{\lambda^{-1}\alpha_k}{2} \norm{\bar y^k - y^k_+}^2,   \notag\\ 
\varkappa_4 \leq &\, \frac{C_\gamma \alpha_k^2}{2} \left\{2\norm{\nabla F(\bar x^k) - \bar z^k}^2 +  \frac{2 L_{\nabla F}^2}{n}\norm{\mX_k - \bar \mX_k}^2 + \norm{\bar \Delta^{k+1}}^2 + 2\<\bar \Delta^{k+1}, \frac{1}{n}\sum_{i=1}^{n} \nabla F_i(x^k_i) - \bar z^k> \right\}. \notag
\end{align}
Combining the above results with \eqref{ineq:zbar-mse-recursive} in Lemma \ref{lem: z_nabla_F} and the definition of $W(\bar x^k, \bar z^k)$ in \eqref{def: merit_fun}, we have
\begin{align}
     &W(\bar x^{k+1}, \bar z^{k+1}) - W(\bar x^{k}, \bar z^{k}) \leq  \alpha_k\left\{- \frac{5}{8}\gamma^{-1} + \frac{(L_{\nabla F} + C_\gamma)\alpha_k}{2} + 2\lambda L_{\nabla F}^2 \right\}\|\bar x^k  - \bar y^k\|^2 \notag \\
     &+ \alpha_k\left\{ - \frac{\lambda}{2} + C_\gamma \alpha_k \right\}\norm{\nabla F(\bar x^k) - \bar z^k}^2 + \frac{C_\gamma \alpha_k^2}{2}\norm{\bar \Delta^{k+1}}^2 +  \frac{(\gamma^{-1}  + 2\lambda^{-1})\alpha_k}{4} \norm{y^k_+ - \bar y^k}^2 \notag \\
    & + \frac{\gamma\alpha_k}{n}\norm{\mZ_k - \bar \mZ_k}^2 + \frac{\left(\gamma^{-1}+ 2\gamma L_{\nabla F}^2 + 2\lambda L_{\nabla F}^2 + C_\gamma L_{\nabla F}^2\alpha_k \right)\alpha_k}{n}\norm{\mX_k - \bar \mX_k}^2 \notag\\
     & +  \alpha_k \underbrace{\< \bar \Delta^{k+1},  \bar x^k - y^k_+ + C_\gamma \alpha_k\left(\frac{1}{n}\sum_{i=1}^{n} \nabla F_i(x^k_i) - \bar z^k\right) + 2\lambda \left( (1-\alpha_k) \left(\bar z^k - \nabla F(\bar x^k)\right) + \alpha_k\delta^{k}\right)>}_{r^{k+1}}. \label{ineq: merit_diff_2}
\end{align}
In addition, from Lemma \ref{lem: consensus_y}, we already know
\begin{equation*}
    \norm{y^k_+ - \bar y^k}^2  \leq \frac{2}{n} \left\{\|\mX_{k} - \bar \mX_k\|^2 + \gamma^2\|\mZ_{k} - \bar \mZ_k\|^2 \right\}.
\end{equation*}
Finally, choosing $\alpha_k$ such that $ \alpha_k \leq \min\{\frac{\gamma^{-1}}{8 L_{\nabla F}}, \frac{\gamma^{-1}}{8 C_\gamma}, \frac{\gamma^{-1}}{32C_\gamma L_{\nabla F}^2}\}$ and $\lambda = \frac{\gamma^{-1}}{8 L_{\nabla F}^2}$, we can re-organize the terms in \eqref{ineq: merit_diff_2} as
\begin{align*}
     &W(\bar x^{k+1}, \bar z^{k+1}) - W(\bar x^{k}, \bar z^{k})  \notag\\
    \leq & - \alpha_k \underbrace{\left\{\frac{\gamma^{-1}}{4} \|\bar x^k  - \bar y^k\|^2 + \frac{\lambda}{4} \norm{\nabla F(\bar x^k) - \bar z^k}^2 \right\}}_{\Theta^k} + \alpha_k^2 \underbrace{\left\{\frac{C_\gamma + 2\lambda}{2}\norm{\bar \Delta^{k+1}}^2\right\}}_{\Lambda^k}\notag + \alpha_k r^k \\
    & + \alpha_k \underbrace{\left\{\frac{2\gamma(1 + 4\gamma^2L_{\nabla F}^2)}{n}\norm{\mZ_k - \bar \mZ_k}^2 + \frac{2\left(\gamma^{-1} + 3\gamma L_{\nabla F}^2 \right)}{n}\norm{\mX_k - \bar \mX_k}^2\right\}}_{\Upsilon^k},
\end{align*}
which completes the proof.
\end{proof}

\section{Discussions}\label{sec: conserror}
In this section we briefly discuss two different functions that measure the consensus violation of vectors among agents. Suppose agent $i$ has $x_i\in\realset^d$, our consensus error (see, e.g., Definition \ref{def: stat&cons}) can be viewed as
\[
    f(x_1, ..., x_n) = \frac{1}{n}\sum_{i=1}^{n}\|x_i - \bar x\|^2,
\]
where $\bar x := \frac{1}{n}\sum_{i=1}^{n}x_i$, while \texttt{SPPDM} in \cite{wang2021distributed} defines (see Eq. (4a), (4b), (5a), (5b), and (41) in \cite{wang2021distributed})
\begin{equation}\label{eq: sppdm_cons}
    \begin{aligned}
        g_W(x_1, ..., x_n) &= \sum_{i\sim j, 1\leq i< j\leq n}\|x_i - x_j\|^2  \\
    &=\frac{1}{2}\sum_{i= j \text{ or } i\sim j}\left(\|x_i-\bar x\|^2 + \|x_j-\bar x\|^2 - 2\<x_i-\bar x, x_j-\bar x>\right) 
    \end{aligned}
\end{equation}
over a connected network whose weighted adjacency matrix (i.e., mixing matrix) is $W$, and the stationarity therein is defined by using $g_W$. $i\sim j$ means agents $i$ and $j$ are neighbors. Note that in general the relationship between $f$ and $g_W$ largely depends on $W$. We consider several special cases:
\begin{itemize}
    \item W is a complete graph. By \eqref{eq: sppdm_cons} we have
    \begin{align*}
        g_W(x_1,...,x_n) &= n\sum_{i=1}^{n}\|x_i - \bar x\|^2 - \<\sum_{i=1}^{n}\left(x_i - \bar x\right), \sum_{j=1}^{n}\left(x_j - \bar x\right)> = n^2 f(x_1,...,x_n).
    \end{align*}
    \item W is a cycle. By \eqref{eq: sppdm_cons} we have
    \[
        g_W(x_1,...,x_n) \leq \sum_{i\sim j, 1\leq i<j\leq n}2\left(\|x_i - \bar x\|^2 + \|x_j - \bar x\|^2\right) = 4nf(x_1,...,x_n).
    \]
    \item W is a simple path such that $i$ and $i+1$ are adjacent for all $1\leq i\leq n-1$, and $x_i = i\in \realset$. Note that in this case we can directly obtain $g_W(x_1,...,x_n) = n - 1$. For $f$ we have
    \[
        f(x_1,...,x_n) = \frac{1}{n}\sum_{i=1}^{n}\left(\frac{n+1}{2} - i\right)^2 = \Theta(n^2),
    \]
    which implies $g_W = \Theta(\frac{f}{n})$.
\end{itemize}
We know from the above examples that the order (in terms of $n$) of $g_W / f$ can range from $\frac{1}{n}$ to $n^2$. Hence these two types of consensus error are not comparable if no additional assumptions are given, and thus we only include \texttt{SPPDM} in the experiments and do not have it in Table \ref{tab:summary}.

\end{document}